%% file: main.tex
\documentclass[a4paper,12pt,twoside]{article}
\usepackage{graphicx}
\usepackage[section]{placeins}
\usepackage[utf8]{inputenc}
\usepackage{amssymb}
\usepackage{amsmath}
\usepackage{amsthm}
\usepackage[linesnumbered,ruled,vlined]{algorithm2e} 
\usepackage{mathtools}
\usepackage[a4paper,text={150mm,255mm},centering,headsep=5mm,footskip=10mm]{geometry}
\usepackage[onehalfspacing]{setspace}
\usepackage[ngerman, english]{babel}
\usepackage{bbm}
\usepackage{grffile} 
\usepackage{verbatim}
\usepackage{caption}
\usepackage{subcaption}
\usepackage{xcolor}
\usepackage{gensymb}
\usepackage{tcolorbox}
\usepackage{enumitem}
\usepackage{colortbl}
\usepackage{fancyhdr}
\usepackage{url}
\usepackage[noadjust]{cite}
\usepackage{titlesec}
\usepackage{graphicx}
\usepackage{colortbl}
\usepackage[]{algorithm2e}
\usepackage{hyperref}
\usepackage{booktabs}
\usepackage{tabularx}
\usepackage{makecell}
\usepackage{multirow}
\usepackage{authblk}

\tcbuselibrary{breakable}

\definecolor{mycolor}{rgb}{0,0,0}
\makeatletter
\newcommand{\greybox}[1]{%
  \setbox0=\hbox{#1}%
  \setlength{\@tempdima}{\dimexpr\wd0+13pt}%
  \begin{tcolorbox}[colframe=mycolor,breakable,boxrule=0.2pt,arc=4pt,colback=white!10,
      left=6pt,right=6pt,top=6pt,bottom=6pt,boxsep=0pt,width=
\textwidth] 
    #1
  \end{tcolorbox}
}

\newtheoremstyle{boldremark}
    {\dimexpr\topsep/2\relax} 
    {\dimexpr\topsep/2\relax} 
    {}          
    {}          
    {\bfseries} 
    {.}         
    {.5em}      
    {}          

\newcommand{\nw}[1]{\textcolor{blue!85!black}{\textbf{NW: }#1}}

\newcommand*{\R}{\mathbb{R}}

\newcommand*{\fatX}{\textbf{X}}

\newcommand*{\arv}{\nu}
\DeclareMathOperator*{\argmin}{arg\,min}

\newtheorem{theorem}{Theorem}
\numberwithin{equation}{section}
\numberwithin{theorem}{section}

\newtheorem{definition}[theorem]{Definition}
\newtheorem{corollary}[theorem]{Corollary}
\newtheorem{lemma}[theorem]{Lemma}
\newtheorem{proposition}[theorem]{Proposition}
\newtheorem{conjecture}[theorem]{Conjecture}
\theoremstyle{boldremark}
\newtheorem{remark}[theorem]{Remark}
\theoremstyle{boldremark}

\usepackage{geometry}
\makeatother

\titleformat{\chapter}[display]
   {\normalfont\huge\bfseries}{ \thechapter}{20pt}{\Huge}

\titlespacing*{\chapter}{0pt}{-40pt}{20pt}
\definecolor{RoyalRed}{RGB}{157,16,45}
\titleformat{\chapter}[display]
  {\bfseries\LARGE}
  {\huge
  \chaptertitlename\hspace{0.1ex} \thechapter}{1pc}
  {{\titlerule[0pt]}\vspace{1pc}}

\titleformat{\chapter}[hang]{\bfseries\huge}{\bfseries\thechapter}{1em}{}

\providecommand{\keywords}[1]
{
  \small	
  \textbf{\text{Keywords---}} #1
}
\usepackage[scaled]{helvet}
\usepackage{courier}

\begin{document}

\title{Measuring dependencies between variables of a dynamical system using fuzzy affiliations}
\author[]{Niklas Wulkow}
\affil[]{Department of Mathematics and Computer Science, Freie Universität Berlin, Germany}
\affil[]{Zuse Institute Berlin, Germany}
\maketitle

\setcounter{tocdepth}{2}

\begin{abstract}
    A statistical, data-driven method is presented that quantifies influences between variables of a dynamical system. The method is based on finding a suitable representation of points by fuzzy affiliations with respect to landmark points using the Scalable Probabilistic Approximation algorithm. This is followed by the construction of a linear mapping between these affiliations for different variables and forward in time. This linear mapping can be directly interpreted in light of unidirectional dependencies and relevant properties of it are quantified. These quantifications then serve as measures for the influences between variables of the dynamics. The validity of the method is demonstrated with theoretical results and on several numerical examples, including real-world basketball player movement.
\end{abstract}

\keywords{Dependency measures, influence detection, data-driven modelling, Scalable Probabilistic Approximation, barycentric coordinates, basketball analytics}\\\\
\textsc{MSC classes: 37M10, 62B10, 65C20}


\input{Introduction}
\input{SPA}

\input{DependencyAnalysis}

\input{NumericalExamples}
\input{BasketballExample}
\input{Conclusion}
\bibliographystyle{alpha}
\bibliography{SPA_Dependencies_BIB}
\input{Appendix}

\end{document}

%% file: Introduction.tex
\section{Introduction}

Over the last decades, detecting influences between variables of a dynamical system from time series data has shed light on the interplay between quantities in various fields, such as between genes~\cite{sincerities}, between wealth and transportation~\cite{yetkiner} or between climate change and marine populations~\cite{nakayama}.
There exist various conceptually different numerical methods which aim to detect relations between variables from data. For example, a prominent method is Convergent Cross Mapping (CCM)~\cite{sugihara}, introduced in 2012, which uses the delay-embedding theorem of Takens~\cite{takens}. 
Another prominent method is Granger causality~\cite{granger,kirchgaessner} which was first used in economics in the 1960s and functions based on the intuition that if a variable $X$ forces $Y$, then values of $X$ should help to predict $Y$. Much simpler than these methods is the well-known Pearson correlation coefficient~\cite{pearson}, introduced already in the 19th century. There exist many more, such as the Mutual Information Criterion~\cite{tourassi}, the Heller-Gorfine test~\cite{heller}, Kendall's $\tau$~\cite{puka}, the transfer entropy method~\cite{ursino} or the Kernel Granger test~\cite{krakovska}.

Recently, in \cite{cliff} a summarising library comprising a broad range of statistical methods was presented together with many examples. This illustrates the large number of already existing statistical techniques to detect influences between variables of a dynamical system.
However, as also pointed out in \cite{cliff}, such methods have different strengths and shortcomings, making them well-suited for different scenarios while ill-suited for others. For example, CCM requires data coming from dynamics on an attracting manifold and suffers when the data are subject to stochastic effects. Granger causality needs an accurate model formulation for the dynamics which can be difficult to find. Pearson's correlation coefficient requires statistical assumptions on the data that are often not met. The authors of \cite{cliff} therefore suggest to utilize many different methods on the same problem instead of only a single one to optimally find and interpret relations between variables.

\bigskip
In this article, a practical method is presented which aims at complementing the weaknesses of related methods and enriching the range of existing techniques for detection of influences. 
It represents the data using probabilistic, or \textit{fuzzy}, affiliations with respect to landmark points and constructs a linear probabilistic model between the affiliations of different variables and forward in time. This linear mapping then admits a direct and intuitive interpretation in regard to the relationship between variables. The landmark points are selected by a data-driven algorithm and the model formulation is quite general so that only very little intuition of the dynamics is required.
For the fuzzy affiliations and the construction of the linear mapping, we use the recently introduced method Scalable Probabilistic Approximation (SPA)~\cite{spa,espa} whose capacity to locally approximate highly nonlinear functions accurately has been demonstrated in \cite{spa}. We then compute properties of this mapping which serve as our dependency measures. 
The intuition, which is further elaborated on in the article, now is as follows: if one variable has little influence on another, the columns of the linear mapping should be similar to each other while if a variable has strong influence, the columns should be very different. We quantify this using two measures which extract quantities of the matrix, one inspired by linear algebra, the other a more statistical one. The former computes the sum of singular values of the matrix with the intuition that a matrix consisting of similar columns is close to a low-rank matrix for which many singular values are zero~\cite{jieping}. The latter uses the statistical variance to quantify the difference between the columns. We prove that they are in line with the above interpretation of the column-stochastic matrix and apply the method to several examples to demonstrate its efficacy. Three examples are of theoretical nature where built-in influences are reconstructed. One real-world example describes the detection of influences between basketball players during a game.
\bigskip

This article is structured as follows: in Section~\ref{sec:method} the SPA method is introduced, including a mathematical formalization of influences between variables in the context of the SPA-induced model. In Section~\ref{sec:dependency} the two dependency measures which quantify properties of the SPA model are introduced and connected with an intuitive perception of the SPA model. In Section~\ref{sec:examples}, the dependency measures are applied to examples.

To outline the structure of the method in advance, the three main steps that are proposed to quantify influences between variables of a dynamical system are (also see Figure~\ref{fig:introDiagram}):
\begin{enumerate}
    \item Representation of points by fuzzy affiliations with respect to landmark points.
    \item Estimation of a particular linear mapping between fuzzy affiliations of two variables and forward in time.
    \item Computation of properties of this matrix.
\end{enumerate}

\begin{figure}[ht]
\centering
\includegraphics[width=\textwidth]{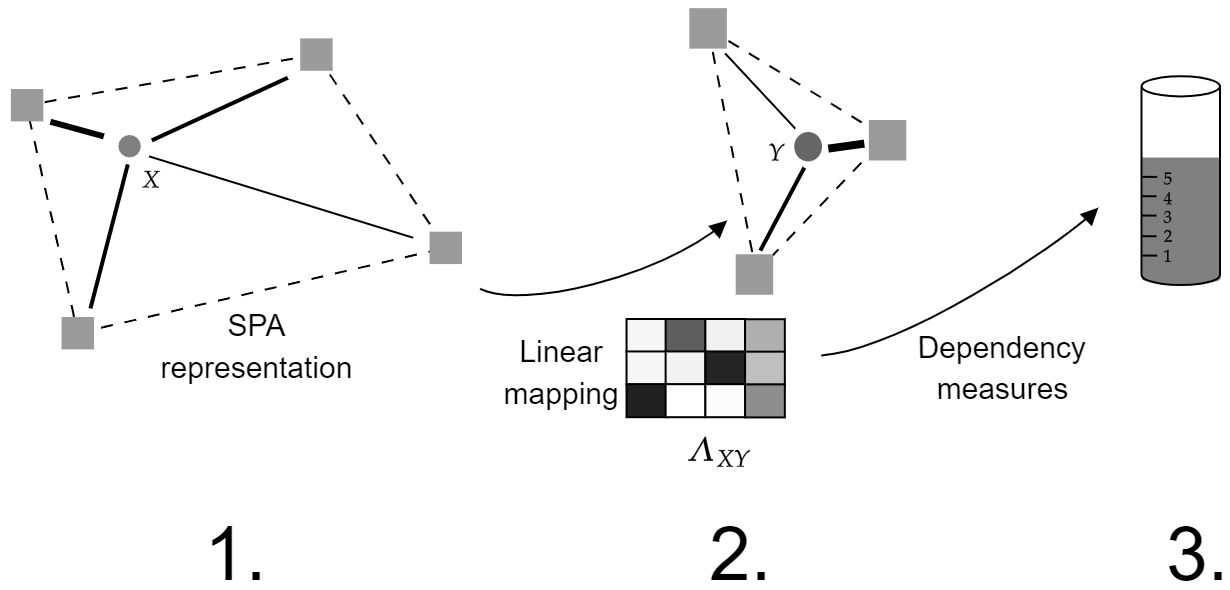}
\caption{Illustration of the three steps of the method presented in this article.}
\label{fig:introDiagram}
\end{figure}

%% file: SPA.tex
\section{Scalable Probabilistic Approximation (SPA)}
\label{sec:method}
Scalable Probabilistic Approximation (SPA)~\cite{spa} is a versatile and low-cost method that transforms points from a $D$-dimensional state space to $K$-dimensional probabilistic, fuzzy, affiliations. If $K$ is less than $D$, SPA serves as a dimension reduction method by representing points as closely as possible. If $K > D$, SPA can be seen as a probabilistic, or fuzzy, clustering method which assigns data points to \textit{landmark points} in $D$-dimensional space depending on their closeness to them. For two different variables $X$ and $Y$, SPA can furthermore find an optimal linear mapping between the probabilistic representations.

The first step, the transformation to fuzzy affiliations, will be called SPA~I while the construction of the mapping in these coordinates will be called SPA~II.

\subsection{Representation of the data in barycentric coordinates}
The mathematical formulation of SPA~I is: let $\fatX = [X_1,\dots,X_T] \in \R^{D\times T}$. Then solve
\begin{equation}
\tag{SPA~I}
\label{eq:SPA1}
\begin{split}
[\Sigma,\Gamma] &= \argmin \Vert \fatX - \Sigma \Gamma \Vert_F\\
\text{subject to } &\Sigma = [\sigma_1,\dots,\sigma_K] \in \R^{D\times K},\quad \Gamma = [\gamma_1,\dots,\gamma_T] \in \R^{K\times T},\\
&(\gamma_t)_i \geq 0, \quad  \sum_{i=1}^K (\gamma_t)_i = 1.
\end{split}
\end{equation}
Is was discussed in \cite{wulkowmSPA} that for $K \leq D$, the representation of points in this way is the orthogonal projection onto a convex polytope with vertices given by the columns of $\Sigma$. The coordinates $\gamma$ then specify the position of this projection with respect to the vertices of the polytope and are called \textit{barycentric coordinates} (BCs). A high entry in such a coordinate then signals closeness of the projected point to the vertex.
\begin{remark}
A similar representation of points has already been introduced in PCCA+~\cite{weber}.
\end{remark}

For $K>D$, however, in \cite{spa} the interpretation of a probabilistic clustering is introduced. According to the authors, the entries of a $K$-dimensional coordinate of a point denotes the probabilities to be inside a certain box around a landmark point. One can generalize this interpretation to \textit{fuzzy affiliations} to these landmark points, again in the sense of closeness. A BC $\gamma$ then denotes the distribution of affiliations to each landmark point and \eqref{eq:SPA1} can be solved without loss, i.e., so that
\begin{equation}
    X_t = \Sigma \gamma_t
    \label{eq:XSigmagamma}
\end{equation}
holds for all data points. Figure~\ref{fig:spabild} shows the representation of a point in $\R^2$ with respect to four landmark points.

\begin{remark}
From now on, regardless of the relation between $K$ and $D$, we will mostly use the term BC instead of "fuzzy affiliations" for the $\gamma_t$ to emphasise that they are new \textit{coordinates} of the data which we can always map back to the original data as long as Eq.~\eqref{eq:XSigmagamma} is fulfilled. Note that while commonly in the literature the term "barycentric coordinate" refers to coordinates with respect to the vertices of a simplex (i.e., is $(K-1)$-dimensional with $K$ vertices, contrary to the assumption $K > D+1$), in various publications, e.g.,~\cite{anisimov}, the term \textit{generalized barycentric coordinates} is used if the polytope is not a simplex. In any case, we will omit the term "generalized" and write "BC".
\end{remark}

\begin{figure}[ht]
\centering
\includegraphics[width=.75\textwidth]{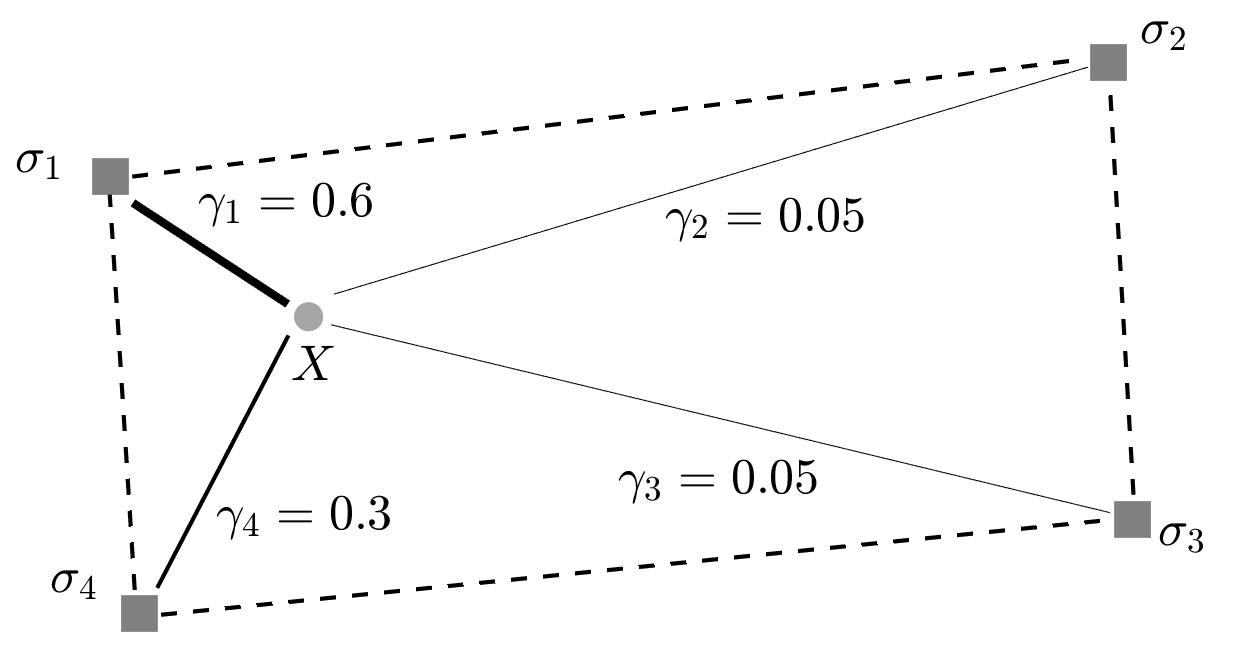}
\caption{Representation of a point $X$ by barycentric coordinates $\gamma_1,\dots,\gamma_4$ with respect to the vertices of a polytope $\sigma_1,\dots,\sigma_4$.}
\label{fig:spabild}
\end{figure}

For $K> D+1$ and given landmark points, the representation with barycentric coordinates is generally not unique (while the set of points that can be represented by $\Sigma \gamma$ is a convex polytope, some landmark points can even lie inside this polytope if $K > D+1$). Therefore, we define the representation of a point $X$ analogously to \cite{wulkowmSPA} in the following way:
\begin{equation}
\begin{split}
    \rho_\Sigma(X,\gamma) &:= \argmin\limits_{\gamma^*} \Vert \gamma - \gamma^* \Vert_2 \\  
    \text{ s.t. }\gamma^* &= \argmin\limits_{\gamma'}\Vert X - \Sigma\gamma'\Vert_2
    \text{ with } \gamma'_\bullet,\gamma^*_\bullet  \geq 0 \text{ and } \Vert \gamma' \Vert_1, \Vert \gamma^* \Vert_1 = 1.
    \end{split}
    \label{eq:rhoK>D}
\end{equation}
$\gamma^*$ should be selected among all barycentric coordinates that represent $X$ without loss so that it is closest to the \textit{reference coordinate} $\gamma$.

\begin{remark}
Note that the solution for the landmark points determined by SPA~I is never unique if $K > 1$~\cite{spa}. In order to solve SPA~I, we iteratively minimize the objective function in \eqref{eq:SPA1} by separately solving for $\Sigma$ and $\Gamma$. For this, we randomly draw initial values. Therefore, for $K>D+1$, an exact solution can be achieved by placing $D+1$ points so that all data points lie inside their convex hull (all convex combinations of them) while the remaining $K-(D+1)$ landmark points can be chosen arbitrarily. As a consequence, the placement of landmark points depends strongly on the randomly chosen initial values of the optimization process. In simple cases, e.g., when the state space of the data is one-dimensional, one could even manually place landmark points across an interval containing the data and solve only for $\Gamma$.
\label{rem:SPA1notunique}
\end{remark}

Let now a dynamical system be given by
\begin{equation}
    X_t = F(X_{t-1}).
\end{equation}
In this case we select $\gamma_{t-1}$ as reference coordinate for the point $X_t$. Using that Eq.~\eqref{eq:XSigmagamma} holds for $K>D$, this gives
\begin{equation}
\gamma_t = \rho_{\Sigma}(X_t,\gamma_{t-1}) = \rho_\Sigma(F(X_{t-1}),\gamma_{t-1}) = \rho_\Sigma(F(\Sigma\gamma_{t-1}),\gamma_{t-1}) =: \textbf{v}(\gamma_{t-1}).
\label{eq:gammadynamics}
\end{equation}
With this, $\gamma_t$ solely depends on $\gamma_{t-1}$. We therefore have formulated a time-discrete dynamical system in the barycentric coordinates. By choosing the reference coordinate as $\gamma_{t-1}$ we assert that the steps taken are as short as possible.

\subsection{Model estimation between two variables}
Assuming that we consider data from two different variables $X \in \R^D$ and $Y \in \R^E$ and strive to investigate the relationship between them, one can solve \eqref{eq:SPA1} for both of them, finding landmark points denoted by the columns of $\Sigma^X \in \R^{D\times K_X},\Sigma^Y\in \R^{E\times K_Y}$ and BCs $\gamma^X_t$ and $\gamma^Y_t$ for $t = 1,\dots,T$ and solve \ref{eq:SPA2}, given by
\begin{equation}
\tag{SPA~II}
\begin{split}
\Lambda &= \argmin_{\Lambda^*\in \R^{K\times K}} \Vert [\gamma^Y_1 \vert \cdots \vert\gamma^Y_T] - \Lambda^* [\gamma^X_1 \vert \cdots\vert\gamma^X_{T}] \Vert_F, \\
\text{subject to } & \Lambda \geq 0 \text{ and } \sum_{k=1}^K \Lambda_{k,\bullet} = 1. 
\end{split}
\label{eq:SPA2}
\end{equation}
$\Lambda$ is a column-stochastic matrix and therefore guarantees that BCs are mapped to BCs again. It is an optimal linear mapping with this property that connects $X$ to $Y$ on the level of the BCs. One can easily transform back to the original state space by
\begin{equation}
    Y_t = \Sigma \gamma^Y_{t} \approx \Sigma \Lambda \gamma^X_{t}.
\end{equation}

Note that is has been demonstrated in \cite{spa} that this linear mapping applied to the BCs can accurately approximate functions of various complexity and is not restricted to linear relationships.

\begin{remark}
In \cite{spa} is shown a way to combine SPA~I and SPA~II in a single SPA~I problem. The landmark points are then selected so that the training error of the SPA~II problem can be set to $0$. In other words, optimal discretisations of the state spaces are determined where optimal means with respect to the exactness of the ensuing \ref{eq:SPA2} model.
\end{remark}

\subsubsection{Estimation of dynamics}
A special case of $\eqref{eq:SPA2}$ is the estimation of dynamics, i.e., solving
\begin{equation}
\begin{split}
\Lambda &= \argmin_{\Lambda^*\in \R^{K\times K}} \Vert [\gamma^X_2 \vert \cdots \vert\gamma^X_T] - \Lambda^* [\gamma^X_1 \vert \cdots\vert\gamma^X_{T-1}] \Vert_F, \\
\text{subject to } & \Lambda \geq 0 \text{ and } \sum_{k=1}^K \Lambda_{k,\bullet} = 1. 
\end{split}
\label{eq:SPA2dynamics}
\end{equation}
$\Lambda$ is therefore a linear, column-stochastic approximation of the function $\textbf{v}$ from Eq.~\eqref{eq:gammadynamics}. Such a matrix is typically used in so-called Markov State Models~\cite{sarich,husic}. With $\Lambda$, we can construct dynamics in the BCs with the possibility to transform back to the original state space by multiplication with $\Sigma$, since
\begin{equation}
    X_{t} = \Sigma \gamma_{t} =  \Sigma\textbf{v}(\gamma_{t-1}) \approx \Sigma \Lambda \gamma_{t-1}.
\end{equation}

\subsection{Forward model estimation between two processes}
Given two dynamical systems
\begin{equation}
    X_t = F(X_{t-1}) \in \R^D \text{ and } Y_t = G(Y_{t-1}) \in \R^E,
\end{equation}
we will now determine a column-stochastic matrix which propagates barycentric coordinates from one variable to the other and forward in time. We will use this mapping to quantify the effect of a variable on future states of the other.

With landmark points in $\Sigma^X \in \R^{D\times K_X},\Sigma^Y\in \R^{E\times K_Y}$ and BCs $\gamma^X_t$ and $\gamma^Y_t$ for $t = 1,\dots,T$, let us find a column-stochastic matrix $\Lambda_{XY}$ that fulfils
\begin{equation}
\begin{split}
\Lambda_{XY} = \argmin_{\Lambda^*\in \R^{K_Y\times K_X}} \Vert [\gamma^Y_2 \vert \cdots \vert\gamma^Y_T] - \Lambda^* [\gamma^X_1 \vert \cdots\vert\gamma^X_{T-1}] \Vert_F.
\end{split}
\label{eq:SPA2XtoY}
\end{equation}
$\Lambda_{XY}$ represents a model from $X_{t-1}$ to $Y_t$ on the level of the BCs and tries to predict subsequent values of $Y$ using only $X$.

Now, let us assume that $X$ in fact has direct influence on $Y$, meaning that there exists a function
\begin{equation}
   H(Y_{t-1},X_{t-1}) = Y_{t}.
   \label{eq:wYX}
\end{equation}
Then similarly as when constructing the dynamical system in the BCs previously in Eq.~\eqref{eq:gammadynamics}, we can observe
\begin{equation}
\begin{split}
\gamma^Y_t &= \rho_{\Sigma^Y}(Y_t,\gamma^Y_{t-1})\\
\text{\eqref{eq:wYX} gives }&= \rho_{\Sigma^Y}(H(Y_{t-1},X_{t-1}),\gamma^Y_{t-1})\\
\text{\eqref{eq:XSigmagamma} yields }&= \rho_{\Sigma^Y}(H(\Sigma^X\gamma^X_{t-1},\Sigma^Y\gamma^Y_{t-1}),\gamma^Y_{t-1})\\
\text{which we define as } &=: \textbf{w}(\gamma^X_{t-1},\gamma^Y_{t-1}).   
\end{split}
\end{equation}
$\gamma^Y_t$ therefore directly depends on $\gamma^X_{t-1}$ and $\gamma^Y_{t-1}$ while $\Lambda_{XY}$ attempts to predict $\gamma_t^Y$ using only $\gamma_{t-1}^X$. Assuming that the approximation
\begin{equation}
    \Lambda_{XY} \gamma_{t-1}^X \approx \gamma^Y_t
\end{equation}
is close for each pair of $\gamma^X_{t-1},\gamma^Y_{t}$, we can assert
\begin{equation}
    \gamma^Y_t \approx \Lambda_{XY} \gamma_{t-1}^X = \sum_{j=1}^{K_X} (\Lambda_{XY})_{|j} (\gamma^X_{t-1})_j,
    \label{eq:LambdaXYmapping}
\end{equation}
where $(\Lambda_{XY})_{|j}$ is the $j$th column of $\Lambda_{XY}$.
A prediction for $\gamma^Y_t$ is therefore constructed using a weighted average of the columns of $\Lambda_{XY}$. The weights are the entries of $\gamma^X_{t-1}$.

\begin{remark}
Note that the same argumentation starting in Eq.~\eqref{eq:SPA2XtoY} holds if we choose a time shift of length $\tau > 0$ and consider information of $X_{t-\tau}$ about $Y_t$. If $\tau =1$ we will simply write $\Lambda_{XY}$ but generally write $\Lambda_{XY}^{(\tau)}$.
\end{remark}

\begin{remark}
Since $\Lambda_{XY}\gamma^X_{t-1}$ estimates $\gamma^Y_t$ using only $\gamma^X_{t-1}$ although it additionally depends on $\gamma^Y_{t-1}$, one can interpret $\Lambda_{XY}$ as an approximation to the conditional expectation of $\gamma^Y_t$ given $\gamma^X_{t-1}$, i.e., assuming that $\gamma_{t-1}^Y$ is distributed by a function $\mu_Y$,
\begin{equation}
    \Lambda_{XY} \gamma \stackrel{}{\approx} \mathbb{E}_{\mu_Y}[\gamma^Y_t | \gamma^X_{t-1} = \gamma].
    \label{eq:condexpLambdaXY}
\end{equation}
In Appendix~\ref{sec:App_condexp}, this intuition is formalized further.
\label{rem:condexp}
\end{remark}

%% file: DependencyAnalysis.tex
\section{Quantification of dependencies between variables}
\label{sec:dependency}
In the following, we will define two methods that quantify the strength of dependence of $Y_{t}$ on $X_{t-1}$ by directly computing properties of $\Lambda_{XY}$.
The intuition can be illustrated as follows: if a variable $X$ is important for the future state of another variable, $Y$, then the multiplication of $\Lambda_{XY}$ with $\gamma^X_{t-1}$ should give very different results depending on which landmark point $X_{t-1}$ is close to, i.e., which weight in $\gamma^X_{t-1}$ is high. Since $\gamma^Y_t$ is composed by a weighted average of the columns of $\Lambda_{XY}$ by Eq.~\eqref{eq:LambdaXYmapping}, this means that the columns of $\Lambda_{XY}$ should be very different from each other. In turn, if $X$ has no influence and carries no information for future states of $Y$, the columns of $\Lambda_{XY}$ should be very similar to each other. In Appendix~\ref{sec:App_condexp}, it is shown that given independence of $\gamma^Y_t$ from $\gamma^X_{t-1}$, all columns of $\Lambda_{XY}$ should be given by the mean of the $\gamma^Y$ in the data. There this is also connected to conditional expectations and the intuition given in Eq.~\eqref{eq:condexpLambdaXY}.

In the extreme case that the columns are actually equal to each other, $\Lambda_{XY}$ would be a rank-1 matrix. If the columns are not equal but similar, $\Lambda_{XY}$ is at least \textit{close to} a rank-1 matrix. We should therefore be able to deduce from the similarity of the columns of $\Lambda_{XY}$ if $X$ could have an influence on $Y$.
This is the main idea behind the dependency measures proposed in this section. The intuition is similar to the notion of predictability of a stochastic model introduced in~\cite{rodrigues}.

\begin{remark}
Note that if there is an intermediate variable $Z$ which is forced by $X$ and forces $Y$ while $X$ does not directly force $Y$, then it is generally difficult to distinguish between the direct and indirect influences. In Section~\ref{sec:SDEexample} an example for such a case is investigated.
\end{remark}

\subsection{The dependency measures}
We now introduce the two measures we will use for the quantification of dependencies between variables. Note that these are not "measures" in the true mathematical sense but the term is rather used as synonymous to "quantifications".
\subsubsection{Schatten-1 norm}
For the first measure, we consider the \textit{Singular Value Decomposition} (SVD)~\cite{golub} of a matrix $\Lambda$, given by $\Lambda = USV^T \in \R^{K_Y\times K_X}$. $S\in\R^{K_Y\times K_X}$ is a matrix which is only non-zero in the entries $(i,i)$ for $i = 1,\dots,min(K_X,K_Y)$ which we denote by $s_1,\dots,s_{min(K_X,K_Y)}\geq 0$. $U\in \R^{K_Y\times K_Y}$ and $V\in\R^{K_X\times K_X}$ fulfil certain orthogonality properties and consist of columns $u_i,v_i$. We can thus write $\Lambda$ as
\begin{equation}
\Lambda = \sum\limits_{i=1}^r u_i v_i^T s_{i}
\label{eq:SVDscalar}
\end{equation}
A classic linear algebra result asserts that $rank(\Lambda) = \#\lbrace s_i > 0 \rbrace$. As a consequence, if some of the $s_{i}$ are close to $0$, then Eq~\eqref{eq:SDE_svd} means that only a small perturbation is sufficient to make $\Lambda$ a matrix of lower rank. Therefore, we use the sum of singular values, the so-called \textit{Schatten-1 norm}~\cite{lefkimmiatis}, as a continuous measure of the rank and thus of the difference in the rows of $\Lambda$.
\begin{definition}[Schatten-1 norm]
Let the SVD of a matrix $\Lambda\in \R^{K_Y\times K_X}$ be given by $\Lambda = USV^T$ with singular values $s_1,\dots,s_{min(K_X,K_Y)}$. Then the \text{Schatten-1 norm} of $\Lambda$ is defined as
\begin{equation}
\Vert \Lambda \Vert_1 := \sum\limits_{i=1}^{min(K_X,K_Y)} s_i.
\end{equation}
\end{definition}

\subsubsection{Average row variance}
As our second dependency measure, we directly quantify the difference of columns of a matrix $\Lambda$ using the mean statistical variance per row. We therefore consider every row and compute the variance between its entries, thereby comparing the columns with respect to this particular row. We then take the mean of these variances across all rows.
\begin{definition}[Average row variance]
For a matrix $\Lambda \in \R^{K_Y\times K_X}$, let $\bar{\Lambda}_{i-}$ denote the mean of the $i$th row of $\Lambda$. Let
\begin{equation*}
\arv_i := \frac{1}{K_X-1}\sum\limits_{j=1}^{K_X} (\Lambda_{ij} - \bar{\Lambda}_{i-})^2
\end{equation*}
be the statistical variance of the entries of the $i$th row. Then we define as the \text{average row variance}
\begin{equation}
\arv(\Lambda) := \frac{1}{K_Y} \sum\limits_{i=1}^{K_Y} \arv_i.
\end{equation}
\end{definition}

We will store the calculated values for $\Vert\cdot \Vert_1$ and $\arv$ in tables, resp. matrices of the form
\begin{equation}
M_{\Vert \cdot \Vert_1} = \begin{pmatrix}
\Vert \Lambda_{XX} \Vert_1 & \Vert\Lambda_{XY}\Vert_1 \\
\Vert\Lambda_{YX} \Vert_1 & \Vert\Lambda_{YY}\Vert_1
\end{pmatrix}, M_{v} = \begin{pmatrix}
\arv(\Lambda_{XX}) & \arv(\Lambda_{XY}) \\
\arv(\Lambda_{YX}) & \arv(\Lambda_{YY})
\end{pmatrix}.
\label{eq:Mmatrix}
\end{equation}
Then for each of these matrices, the property $M - M^T$ should be interesting for us, because they contain the differences between dependency measures, stating how strongly $X$ depends on $Y$ compared to $Y$ depending on $X$. We therefore define
\begin{definition}[Relative difference between dependencies]
Let $M$ be one of the matrices from Eq.~\eqref{eq:Mmatrix}. The relative difference between dependencies in both directions is defined as
\begin{equation}
\delta(M)_{ij} = \frac{M_{ij} - M_{ji}}{max(M_{ij},M_{ji})}.
\end{equation}
\end{definition}

\begin{remark}
In Appendix~\ref{sec:differences} it is explained how this method differs from linear correlations, Granger causality and the same approach using boxed instead of fuzzy affiliations .
\end{remark}

When using the dependency measures to analyse which of two variables more strongly depends on the other, it is unclear at this point whether the dimension of the variables and the number of landmark points affects the outcome of the analysis. Hence, for the numerical examples in the next section, we will always consider pairs of variables which have the same dimension and use the same number of landmark points for them, i.e., impose $K_X = K_Y$ in each example to make the comparison a fair as possible. In this light, the following theoretical results on the Schatten-1 norm and the average row variance are directly helpful.

\subsection{Maximizers and minimizers of the dependency measures}
About $\Vert \cdot \Vert_1$ and $\arv$, we can prove properties that validate why they represent sensible measures for the strengths of dependency between two processes. For this we need the following definition.
\begin{definition}[Permutation matrix]
As a \text{permutation matrix} we define a matrix $A \in \lbrace 0,1 \rbrace^{n\times n}$ such that every row and column contains exactly one $1$.
\end{definition}
Then we obtain the following results on the maximisers and minimisers of the Schatten-1 norm and average row variance (half of them only for $K_Y \geq K_X$ or $K_Y=K_X$) whose proofs can be found in Appendix~\ref{sec:proofs}.

\begin{proposition}[Maximal Schatten-1 norm, $K_Y\geq K_X$]
Let $\Lambda \in \R^{K_Y\times K_X}$ with $K_Y \geq K_X$. Then the Schatten-1 norm of $\Lambda$ obtains the maximal value $K_X$ if
deletion of $K_Y-K_X$ rows of $\Lambda$ yields a $K_X \times K_X$ permutation matrix.
\label{lem:schattenminimalKY>KX}
\end{proposition}

\begin{proposition}[Minimal Schatten-1 norm]
The Schatten-1 norm of a column stochastic $(K_Y\times K_X)$-matrix $A$ is minimal if and only if $A_{ij} \equiv \frac{1}{n}$ and its minimal value is equal to $1$.
\label{lem:schattenminimizer}
\end{proposition}
For the average row variance, we can derive the following results:

\begin{proposition}[Maximal average row variance, $K_Y= K_X$]
The average row variance of a column-stochastic $(K_Y\times K_X)$-matrix $\Lambda$ with $K_Y = K_X$ obtains the maximal value $\frac{1}{K_Y}$ if it is a $K_X \times K_X$ permutation matrix. 
\label{lem:variancemaximizer_KY>KX}
\end{proposition}
It seems likely that for $K_Y > K_X$ the the maximers of the Schatten-1 norm from Proposition~\ref{lem:schattenminimalKY>KX} also maximize the average row variance with maximal value $\frac{1}{K_Y}$.

\begin{proposition}[Minimal average row variance]
The average row variance of a column-stochastic $(K_Y\times K_X)$-matrix $\Lambda$ obtains the minimal value $0$ if and only if all columns are equal to each other.
\label{lem:varianceminimizer}
\end{proposition}
In order to analyse the dependencies between two variables for which we use different numbers of landmarks, i.e., $K_X \neq K_Y$, it would be desirable if similar results as above could be inferred for the case $K_Y < K_X$ so that we could make valid interpretations of both $\Lambda_{XY}$ and $\Lambda_{YX}$. However, it was more difficult to prove them so that only the following conjectures are made for the case $K_Y < K_X$:
\begin{conjecture}[Maximal Schatten-1 norm, $K_Y< K_X$]
The Schatten-1 norm of a column-stochastic $(K_Y\times K_X)$-matrix $\Lambda$ with $K_Y < K_X$  is maximal if and only if $\Lambda$ contains a $K_Y\times K_Y$ permutation matrix and the matrix of the remaining $K_X-K_Y$ columns can be extended by $K_Y$ columns to a permutation matrix.
\label{lem:schattenminimalKY<KX}
\end{conjecture}

\begin{conjecture}[Maximal average row variance, $K_Y< K_X$]
The average row variance of a column-stochastic $(K_Y\times K_X)$-matrix $\Lambda$ with $K_Y < K_X$ is maximal if and only if $\Lambda$ contains an $K_Y\times K_Y$ permutation matrix and the matrix of the remaining $K_X-K_Y$ columns can be extended by $K_Y$ columns to a permutation matrix.
\label{lem:variancemaximizer_KY<KX}
\end{conjecture}

In summary, the maximizing and minimizing matrices of $\Vert \cdot \Vert$ and $\arv$ are identical and are of the following forms:
\begin{equation}
\text{Maximal: }
\begin{pmatrix}
0 & 1 & 0\\
0 & 0 & 1\\
1 & 0 & 0 \\
0 & 0 & 0
\end{pmatrix},
\quad \text{Minimal:} 
\begin{pmatrix}
\frac{1}{n} & \dots & \frac{1}{n}\\
\vdots & & \vdots\\
\frac{1}{n} & \dots & \frac{1}{n}
\end{pmatrix}.
\end{equation}
These results show that the two dependency measures $\Vert \cdot \Vert_1$ and $\arv$ fulfil important intuitions: they are minimal, when information about $X_{t-1}$ gives us no information about $Y_t$ because in this case all columns of $\Lambda$ should be identical and even equal to each other. Maximal dependence is detected if the information about $X_{t-1}$ is maximally influential to $Y_t$. This happens when $\Lambda$ is, respectively can be reduced or extended to, a permutation matrix. This is illustrated in Figure~\ref{fig:spa_dependencies_bild}.
\begin{figure}[ht]
\centering
\includegraphics[width=\textwidth]{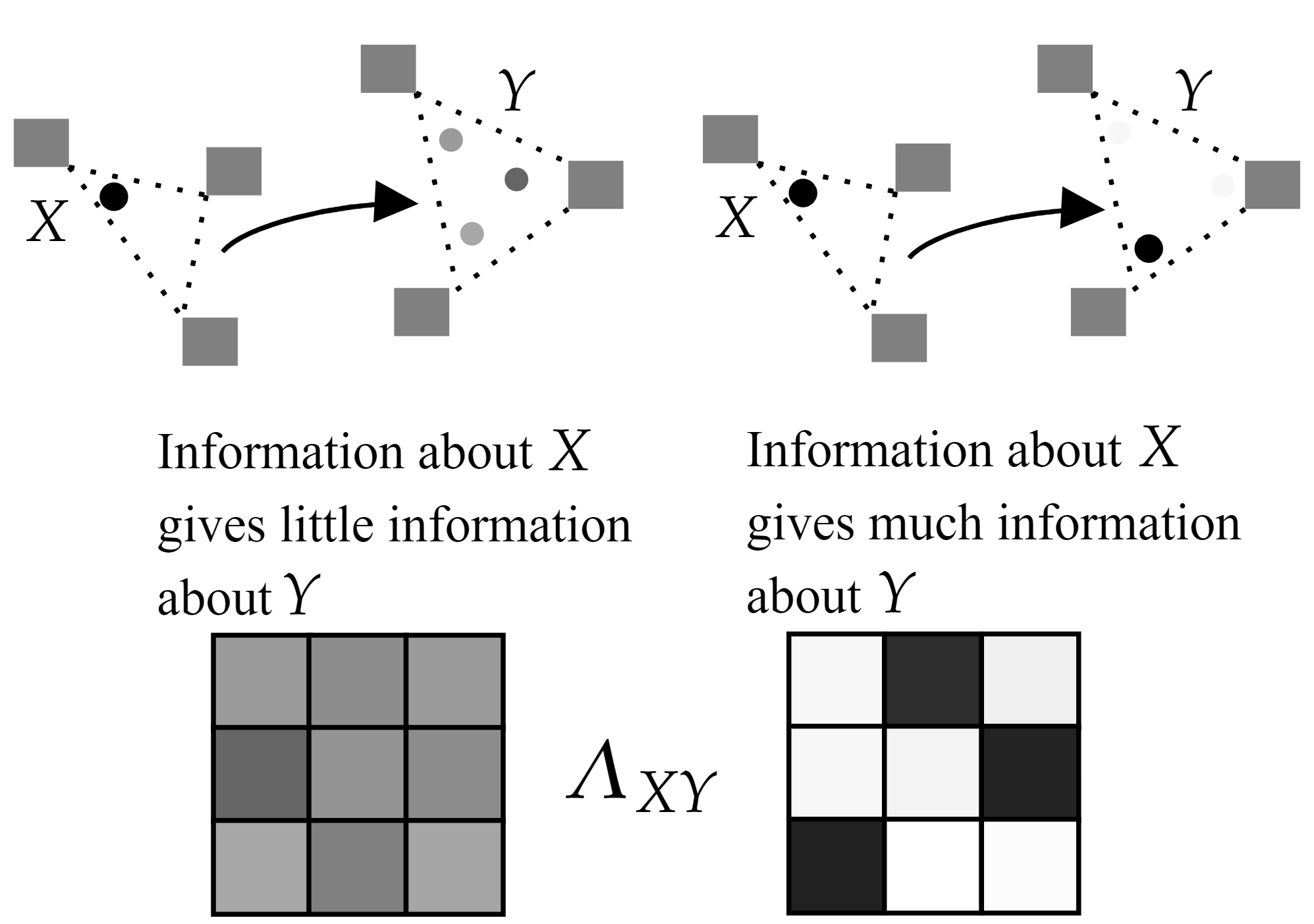}
\caption{Illustration of the intuition behind the dependency measures. If the distribution of $Y$ is independent of $X$, the matrix $\Lambda_{XY}$ will have similar columns (left). If the distribution of $Y$ is strongly dependent of $X$, the columns of $Y$ will be very different from each other.} 
\label{fig:spa_dependencies_bild}
\end{figure}

%% file: NumericalExamples.tex
\section{Numerical examples}
\label{sec:examples}
We now demonstrate the two dependency measures on examples of dynamical systems and aim to cover different structures and properties. In order to assess their efficacy, we explicitly install unidirectional dependencies in the formulation of the dynamics and investigate whether these are detected. For the usage of the SPA method, we use the Matlab code provided on \url{https://github.com/SusanneGerber/SPA}. For two examples we briefly compare our results with those of Convergent Cross Mapping (CCM) for which we use a Matlab implementation provided in \cite{jakubik}.

\subsection{Low-dimensional and deterministic: coupled two-species logistic map}
The first example is a coupled two-species logistic difference system as used in \cite{sugihara}. It describes oscillating behaviour of two quantities while each is forced by the other, although with different strengths. The behaviour of each variable is mainly determined by the logistic map~\cite{may} which for one variable reads
\begin{equation}
X_{t} = rX_{t-1}(1 - X_{t-1}).
\label{eq:logisticmap}
\end{equation}
There is extensive literature on the dependence of the behaviour of the system on the parameter $r$, e.g.,~\cite{tsuchuya, nagashima, may}. For $r < 3$, the system will converge to the value $\frac{r-1}{r}$ while for $r > 4$ the system diverges for most initial values. The most interesting behaviour can be observed for for $r \in (3,4)$. Then the system converges to a set of points which it oscillates between and whose number increases with $r$ approaching $4$. 
\paragraph{The dynamics}
We define a coupled two-species logistic map, taken from \cite{sugihara}, as
\begin{equation}
\begin{split}
X_{t} &= 3.8 X_{t-1}(1-X_{t-1}) - 0.02 Y_{t-1}X_{t-1}\\
Y_{t} &= 3.5 Y_{t-1}(1-Y_{t-1}) - 0.1X_{t-1} Y_{t-1}
\end{split}
\label{eq:sugihara}
\end{equation}
with $X_0 = Y_0 = 0.8$.

Through the prefactor $0.1$ of $X_{t-1}$ in the equation for $Y_{t}$ compared to the lower prefactor $0.02$ of $Y_{t-1}$ in the equation for $X_{t}$, $X$ has a direct influence on $Y$ that is higher than the direct influence that flows in the other direction. We will uncover this using SPA and the dependency measures on the matrices $\Lambda_{XY}$ and $\Lambda_{YX}$. It was already shown in \cite{sugihara} that $X$ and $Y$ are not linearly correlated while CCM can uncover the one-directional influence between the variables.
\paragraph{Dependency analysis}
For the analysis, we use $K_X = K_Y = 10$ landmark points for both processes which lie in and around the interval $[0,1]$ in which the dynamics occur. For a simple case such as this one, one could have even chosen the landmark points manually, for example as $\Sigma = [0,0.1,\dots,1]$ but we use the result produced by the numerical solution of $\eqref{eq:SPA1}$ (in this light, please remember Remark~\ref{rem:SPA1notunique}). We use realisations of length 1800 to produce these 10-dimensional representations of points in barycentric coordinates and estimate the optimal linear mappings $\Lambda_{XY}$ and $\Lambda_{YX}$. We then perform the dependency analysis explained in the previous section with a time shift of one time step, i.e., $\tau = 1$, starting on the first 200 and increasing by increments of 200 up to all 1800 data points. We see in Figure~\ref{fig:sugiharabeispiel} that the two dependency measures detect the different strengths of influences since the blue line denoting the influence of $X$ on $Y$ is consistently above the red one denoting the influence in the opposite direction. The average row variance manages to bring this out more clearly than the Schatten-1 norm.
\begin{figure}[ht]
\centering
\includegraphics[width=\textwidth]{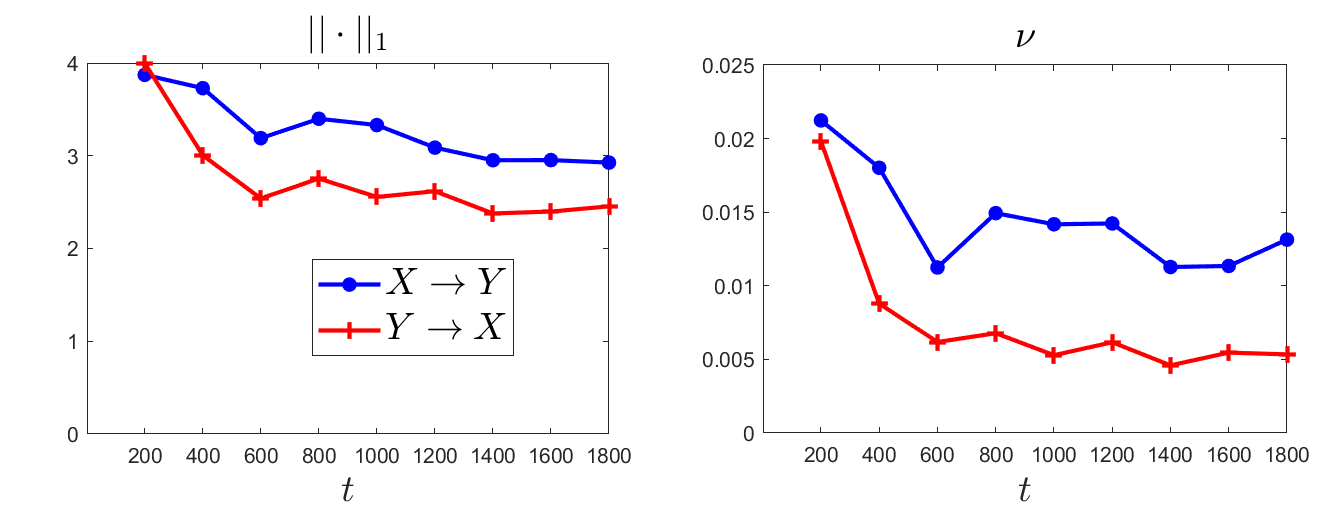}
\caption{Dependency measures $\Vert \cdot \Vert_1$ (left) and $\arv$ (right) for the logistic map system Eq.~\eqref{eq:sugihara} for increasing number of data points $t$ that are taken into account. The dependency measures of $Y$ on $X$ are coloured in blue and the strength of dependence of $X$ on $Y$ in red. We see that both measures uncover the higher influence of $X$ to $Y$ that is anchored in the system while the average row variance gives a clearer distinction than the Schatten-1 norm.} 
\label{fig:sugiharabeispiel}
\end{figure}
For the trajectory of length $1800$, we obtain the result 
\begin{equation*}
M_{\Vert \cdot \Vert_1}:
\begin{tabular}{|l|l|l|l|l|l|}
\hline
From $\downarrow$  to  $\rightarrow$ &  $X_{t+1}$ & $Y_{t+1}$\\
\hline
$X_t$ & 5.76 & 2.92 \\
\hline
$Y_t$ & 2.45 & 6.09 \\
\hline
\end{tabular} 
\Rightarrow
\delta(M_{\Vert \cdot \Vert_1}) = 
 \begin{pmatrix}
\colorbox[rgb]{1.00,1.00,0.5}{0}&\colorbox[rgb]{0.00,1.00,0.5}{0.16}\\\
\colorbox[rgb]{1.00,0.00,0.5}{-0.16}&\colorbox[rgb]{1.00,1.00,0.5}{0}\\
\end{pmatrix}\\
\end{equation*}
\begin{equation*}
M_\arv:
\begin{tabular}{|l|l|l|l|l|l|}
\hline
From $\downarrow$  to  $\rightarrow$ & $X_{t+1}$ & $Y_{t+1}$\\
\hline
$X_t$ & 0.0521 & 0.0131 \\
\hline
$Y_t$ & 0.0053 & 0.0542 \\
\hline
\end{tabular} \Rightarrow
\delta(M_{\arv})
\begin{pmatrix}
\colorbox[rgb]{1.00,1.00,0.5}{0}&\colorbox[rgb]{-0.00,1.00,0.5}{0.6}\\\
\colorbox[rgb]{1.00,-0.00,0.5}{-0.6}&\colorbox[rgb]{1.00,1.00,0.5}{0}\\
\end{pmatrix}.
\label{eq:lmresults}
\end{equation*}
While the Schatten-1 norm produces as relative difference of $0.16$, the average row variance emphasises the stronger dependence of $Y$ on $X$ by a relative difference in dependencies of $0.6$.

\subsection{Low-dimensional and stochastic: Continuous movement by stochastic diffusions}
\label{sec:SDEexample}
This model describes a continuous evolution of processes $A,B,C$ along solutions of a stochastic differential equation (SDE)~\cite{oeksendal} where $C$ acts autonomously and $A$ and $B$ hierarchically depend on each other. In short:
\FloatBarrier
\begin{figure}[ht]
\centering
\includegraphics[width=0.5\textwidth]{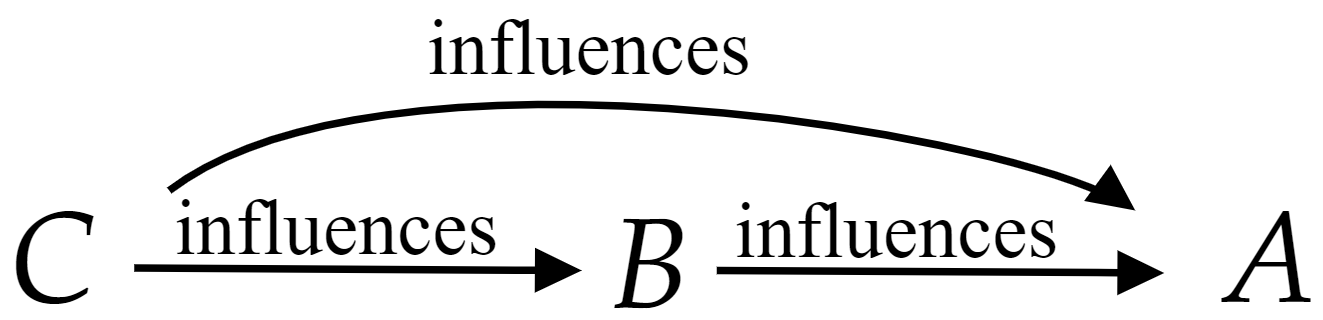}
\end{figure}
\FloatBarrier
\paragraph{The dynamics}
The SDE is given by
\begin{equation}
\begin{split}
dC_t &= G(C_t)~dt + \sigma_C ~dW_{Ct}\\
dB_t &= (\alpha H(B_t) - 10(B_t - C_t))~dt + \sigma_B Id~dW_{Bt}\\
dA_t &= (\alpha H(A_t) - 5(A_t - B_t)-5(A_t-C_t))~dt + \sigma_A Id~ dW_{At}\\
\end{split}
\label{eq:synthModel2}
\end{equation}
where
\begin{equation}
  G(x) = \begin{pmatrix}
0\\
-10(x_2^3-x_2)
\end{pmatrix},\quad H(x) = \begin{pmatrix}
-(x_1^3-x_1)\\
-1
\end{pmatrix}  
\end{equation}
and $Id$ is the identity matrix.

The choice of the potential function $G$ ensures a metastable behaviour of $C$ between regions around the values $-1$ and $1$ in the $x_2$-coordinate, meaning that it remains around one of the values for some time and due to the randomness in the SDE at some point in time suddenly moves towards the other value. In $x_1$-direction, the movement of $C$ is determined by noise with covariance matrix $\sigma_C$. The movements of the other processes are partly governed by the function $H$, which gives a metastable behaviour in the $x_1$-coordinates, and by a difference function between themselves and $C$, respectively (Figure~\ref{fig:SDE} for $\sigma_A=\sigma_B = 0.01$). The movement of $A$ depends equally on its difference to $C$ as on its difference to $B$. Since diffusion processes are attracted to low values of their governing potential, this urges $B$ to move into the direction of $C$. Furthermore, it urges $A$ into the directions of both $C$ and $B$. The parameter $\alpha$ therefore governs how autonomous the processes $A$ and $B$ are.

\begin{figure}[ht]
\centering
\includegraphics[width=\textwidth]{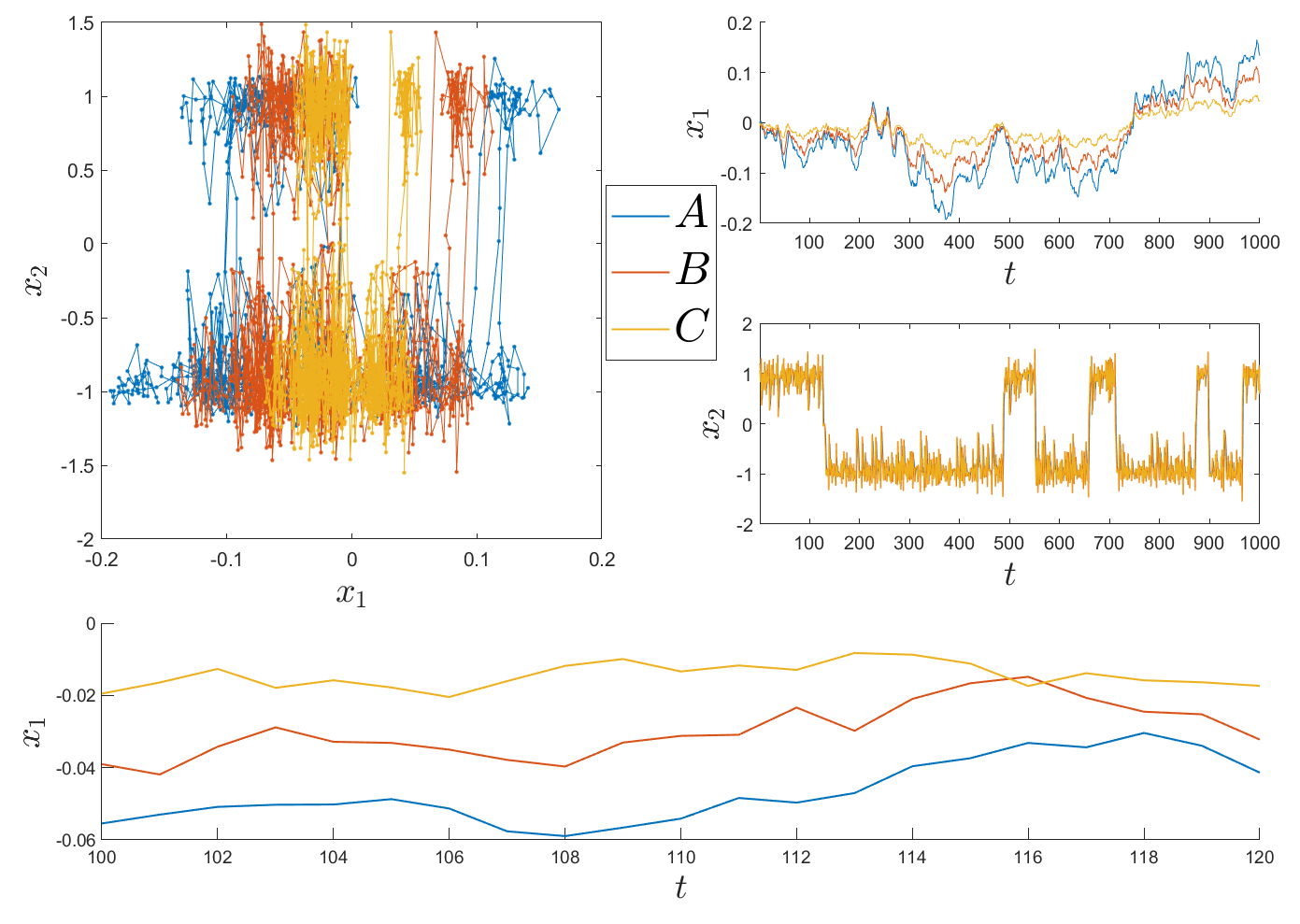}
\caption{Realisation of the system described by Eq.~\eqref{eq:synthModel2}. $A$ in blue, $B$ in orange, $C$ in yellow. We see the metastable behaviour in the $x_2$-coordinate in $C$ that the other processes emulate. Parameters: $\alpha = 5$, $\sigma_C = diag((0.01,0.05))$, $\sigma_B = \sigma_A = 0.01$.}
\label{fig:SDE}
\end{figure}

We create realisations of the processes with the Euler-Maruyama scheme~\cite{kloeden} with a time step of size $\Delta t = 0.1$ for $1000$ time steps. The parameters we use are $\alpha = 5$ and $\sigma_C = diag((0.01,0.05))$. For the noise in the evolution of $A$ and $B$, we use multiple different values $\sigma_B = \sigma_A = 0.01,0.2,0.5$ and $1$.

\paragraph{Dependency analysis}
For SPA~I, we use $K_X = K_Y = 10$. Contrary to the previous example, we compute here the dependencies of the time-differences $\Delta A, \Delta B, \Delta C$ between time steps. Instead of $\Lambda_{XY}^{(\Delta t)}$ we therefore compute $\Lambda_{\Delta X \Delta Y}^{(\Delta t)}$ for $A,B,C = X,Y$. It was also attempted to reconstruct the dependencies using the process values instead of the time-differences but, while one could still reconstruct directions of influences, this came with a lower emphasis in the quantitative results. This presumably is the case since the SDE governs not the next state but rather the next time-differenced value of each process, making the time-differences a better choice here. 

For different values for the noise variances $\sigma_A,\sigma_B$, we create 50 realisations each of the SDE Eq.~\eqref{eq:synthModel2} with the same initial conditions and perform the analysis for each. The SPA~I solution is always computed anew, so that these solutions are generally different from each other, to test the robustness of the method with regard to the SPA~I solution.

The results of the dependency analysis well reflect the hierarchical dependencies between the three processes. An exemplary result is given in Eq.~\eqref{eq:SDE_svd} and Eq.~\eqref{eq:SDE_arv}. The statistics of the overall analysis are given in Table~\ref{tab:SDEresults}. It shows that in the vast majority of the realisations the less influential out of two variables was correctly measured as such. The average row variance $\arv$ gives more clear results with the minimal relative difference at least $0.19$ for large noise but generally around $0.4$. For the Schatten-1 norm, the relative differences are mostly around $0.2$. Note that the results are not strongly influenced by the strength of noise which loosens the strict dependence of $A$ and $B$ on $C$, except for the dependencies between $A$ and $C$ for $\sigma_A = \sigma_B = 1$ for which often $A$ is falsely measured to have the stronger influence between the two.

It was also checked whether the Pearson correlation coefficient was able to detect the directional influences between the variables. It turned out that for $\sigma_A,\sigma_B = 0.01$, the time-differences between the variables where highly correlated according to the imposed hierarchy, e.g., $\Delta A_{t+\Delta t}$ was highly correlated with $\Delta B_t$, but for $\sigma_A,\sigma_B \geq 0.2$ mostly correlation coefficients below $0.1$ could be observed, indicating that the correlation coefficient is not well suited to discover directional dependencies for these dynamics while our dependency measures could do so. CCM does not manage to correctly identify the direction of dependencies in most realisations of the SDE. By construction, the strengths of CCM rather lie in the domain of deterministic dynamics on an attracting manifold (since it uses the delay-embedding theorem of Takens~\cite{takens}) but it is less suited for stochastic dynamics such as this SDE.
\begin{equation}
M_{\Vert \cdot \Vert_1}:
\begin{tabular}{|l|l|l|l|l|l|}
\hline
From $\downarrow$  to  $\rightarrow$ &  $\Delta A_{t+\Delta t}$ & $\Delta B_{t+\Delta t}$ & $\Delta C_{t+\Delta t}$\\
\hline
$\Delta A_t$ & 5.51 & 3.86 & 3.46 \\
\hline
$\Delta B_t$ & 5.29 & 4.11 & 2.94 \\
\hline
$\Delta C_t$ & 4.51 & 7.72 & 4.26\\
\hline
\end{tabular} 
\Rightarrow
\delta(M_{\Vert \cdot \Vert_1}) =
\begin{pmatrix}
\colorbox[rgb]{1.00,1.00,1}{0}&\colorbox[rgb]{1.00,0.71,0}{-0.27}&\colorbox[rgb]{1.00,0.77,0}{-0.23}\\\
\colorbox[rgb]{0.77,1.00,0}{0.27}&\colorbox[rgb]{1.00,1.00,1}{0}&\colorbox[rgb]{1.00,0,0.5}{-0.61}\\\
\colorbox[rgb]{0.77,1.00,0}{0.23}&\colorbox[rgb]{0,1.00,0}{0.61}&\colorbox[rgb]{1.00,1.00,1}{0}
\end{pmatrix}
\label{eq:SDE_svd}
\end{equation}
\begin{equation}
M_\arv:
\begin{tabular}{|l|l|l|l|l|l|}
\hline
From $\downarrow$  to  $\rightarrow$ &  $\Delta A_{t+\Delta t}$ & $\Delta B_{t+\Delta t}$ & $\Delta C_{t+\Delta t}$\\
\hline
$\Delta A_t$ &  0.039 & 0.021 & 0.018 \\
\hline
$\Delta B_t$ & 0.046 & 0.024 & 0.016 \\
\hline
$\Delta C_t$ & 0.039 & 0.065 & 0.026\\
\hline
\end{tabular} \Rightarrow
\delta(M_\arv) =
\begin{pmatrix}
\colorbox[rgb]{1.00,1.00,1}{0}&\colorbox[rgb]{1.00,0.49,0}{-0.54}&\colorbox[rgb]{1.00,0.57,0}{-0.53}\\\
\colorbox[rgb]{0.49,1.00,0}{0.54}&\colorbox[rgb]{1.00,1.00,1}{0}&\colorbox[rgb]{1.00,0.00,0}{-0.75}\\\
\colorbox[rgb]{0.57,1.00,0}{0.53}&\colorbox[rgb]{0.00,1.00,0}{0.75}&\colorbox[rgb]{1.00,1.00,1}{0}\
\end{pmatrix}
\label{eq:SDE_arv}
\end{equation}

\begin{table}[htbp]
\centering
\begin{tabular}{|c|c|c|c|c|c|}
\hline
$\sigma_A,\sigma_B$ & Variables & Average $\delta(M_{\Vert \cdot \Vert_1})$ &  Incorrect & Average $\delta(M_\arv)$ & Incorrect \\
\hline
\multirow{3}{*}{$0.01$} & $A,B$ & -0.21 & 0.12 & -0.36 & 0.2  \\
& $ A,C$ & -0.22 & 0.08 & -0.56 & 0.08  \\ 
& $ B,C$ & -0.46 & 0 & -0.83 & 0.00  \\
\hline
\multirow{3}{*}{$0.2$} & $A,B$ & -0.2 & 0.06 & -0.4 & 0.1  \\
& $ A,C$ & -0.17 & 0.16 & -0.51 & 0.14  \\ 
& $ B,C$ & -0.36 & 0 & -0.76 & 0  \\
\hline
\multirow{3}{*}{$0.5$} & $A,B$ & -0.2 & 0 & -0.4 & 0.02  \\
& $ A,C$ & -0.12 & 0.14 & -0.43 & 0.12  \\ 
& $ B,C$ & -0.30 & 0.02 & -0.71 & 0.04  \\
\hline
\multirow{3}{*}{$1$} & $A,B$ & -0.25 & 0.02 & -0.65 & 0  \\
& $ A,C$ & -0.04 & 0.28 & -0.19 & 0.30  \\ 
& $ B,C$ & -0.22 & 0.08 & -0.62 & 0.02  \\
\hline
\end{tabular}
\caption{Results of the dependency analysis for the diffusion processes in Eq.\eqref{eq:synthModel2}. The third and fifth columns denote the average relative difference between the first variable (in the first row $A$) and the second variable (in the first row $B$) in the Schatten-1 norm and the average row variance. It is always negative, meaning that the first variable (with lower influence) was on average correctly measured as less influential than the second variable. The fourth and sixth columns denote the relative number of occurrences when one variable was falsely identified as more influential, e.g., $A$ as more influential than $B$.}
\label{tab:SDEresults}
\end{table}

\FloatBarrier
\subsection{Higher-dimensional, stochastic and delayed: Multidimensional autoregressive processes}
In order to demonstrate that influence can be detected for processes whose evolution depends not only on present but on past terms, too, we simulate realisations of multidimensional linear autoregressive processes (AR) \cite{brockwell,shumway} in which some variables are coupled with others. An $n$-dimensional linear AR($p$) process is a dynamical system of the form
\begin{equation}
X_{t} = \sum\limits_{i=1}^{p} \phi_i X_{t-1} + \varepsilon_{t}
\end{equation}
where the $\phi_i\in \R^{n\times n}$ and $\varepsilon_{t}$ is a stochastic term which we will set to be normally distributed with mean $0$ and (positive semi-definite) covariance matrix $C\in \R^{n\times n}$.\\
\paragraph{The dynamics}
We now consider AR($p$) processes of the form
\begin{equation}
\begin{pmatrix}
X_{t}\\
Y_{t}
\end{pmatrix}
= \sum\limits_{i=1}^{q}\begin{pmatrix}
\phi_i^{XX} & \phi_i^{YX}\\
0 & \phi_i^{YY}
\end{pmatrix} \begin{pmatrix}
X_{t-i}\\
Y_{t-i}
\end{pmatrix} 
+
\varepsilon_{t}.
\label{eq:arprocess}
\end{equation}
Specifically, we let $X$ and $Y$ be variables in $\R^4$. Thus, the $\phi_i^{XX},\phi_i^{YX}$ and $\phi_i^{YY}$ are matrices in $\R^{4\times 4}$. Through the structure of the coefficient matrices we impose that $X$ is influenced by $Y$ but not vice versa. We let $p = 3$ and construct the $\phi_i$ randomly by drawing for each entry a normally distributed value with mean $0$ and variance $\sigma_i$ where $\sigma_1 = 0.1, \sigma_2 = 0.05, \sigma_3 = 0.03$. This should lead to the coefficients gradually decreasing in magnitude for increasing time lag so that the most recent process values should have the most influence on future ones. $\varepsilon_{t}$ is normally distributed with mean $0$ and covariance matrix $0.01Id$ (where $Id$ is the identity matrix). We then create a realisation of length $T = 1000$ for such a process. This procedure is executed 50 times and the dependency analysis performed on each realisation.

\paragraph{Dependency analysis} 
We choose $K_X = K_Y = 10$ and compute $\Lambda_{XY}^{(\tau)}$ and $\Lambda_{YX}^{(\tau)}$ for $\tau = 1,3,10,50$ to investigate how the dependence between the processes evolves with increasing time shift.

We can see in Table~\ref{tab:ARresults} that the stronger influence of $Y$ on $X$ is recovered for $\tau = 1$ and $\tau = 3$. For $\tau = 1$ the relative differences are stronger which is in line with the fact that the AR coefficients $\phi_3$ were selected to be smaller in magnitude than for $\phi_1$ so that $X_t$ should be more strongly influenced by $Y_{t-1}$ than by $Y_{t-3}$. Moreover, not only the relative differences are smaller for $\tau = 3$ but also the average absolute number of the measures, again correctly indicating a smaller cross-influence with bigger time shift. For $\tau \geq 10$, only negligible differences can be seen. This is consistent with the construction of the processes which include direct influence up to $\tau = 3$.

Again we checked if the Pearson correlation coefficient or CCM could recover the unidirectional dependencies but obtained negative results: the correlation coefficient was again below $0.1$ for most realisations and did not indicate significantly stronger correlation in either direction. CCM indicated only very weak and no directional influence between the variables. By construction, the Granger causality method should perform very well here since the way the dynamics are defined in Eq.~\eqref{eq:arprocess} is perfectly in line with the assumptions of the Granger method. This, of course, cannot be expected for dynamics with unknown model formulation, such as in the following real-world example.

\begin{table}[htbp]
\centering
\begin{tabular}{|c|c|c|c|c|c|}
\hline
$\tau$ & Direction & Average $\Vert \cdot \Vert_1$ & Average  $\delta(M_{\Vert \cdot \Vert_1})$ &  Average $\arv\cdot 10^{2}$  & Average $\delta(M_{\arv})$ \\ \hline
\multirow{2}{*}{1} & $X_{t-1}\rightarrow Y_{t}$ & 1.94 &  -0.19  &  1.67  & -0.58  \\
& $Y_{t-1}\rightarrow X_{t}$ & 2.43 & 0.19 & 4.40  &  0.58  \\ 
\hline
\multirow{2}{*}{3} & $X_{t-3}\rightarrow Y_{t}$ &  1.88  & -0.15 &   1.43 &  -0.49  \\
& $Y_{t-3}\rightarrow X_{t}$ & 2.22 & 0.15 & 3.20 & 0.49  \\ 
\hline
\multirow{2}{*}{10} & $X_{t-10}\rightarrow Y_{t}$ & 1.82 &  -0.02  &  1.27  & -0.06  \\
& $Y_{t-10}\rightarrow X_{t}$ &  1.86 &  0.02 & 1.38 & 0.06  \\ 
\hline
\multirow{2}{*}{50} & $X_{t-50}\rightarrow Y_{t}$ & 1.86  &  0.01  &  1.39 &   0.02  \\
& $Y_{t-50}\rightarrow X_{t}$ & 1.86 &  -0.01  &  1.36  & -0.02  \\ 
\hline
\end{tabular}
\caption{Results for dependency measures of $X$ on $Y$ and vice versa for 50 realisations of processes of the form Eq.~\eqref{eq:arprocess}. The third and fifth columns represents the average absolute value of the Schatten-1 norm and average row variance. They decrease with increasing time lag $\tau$ which is consistent with the reconstruction. The fourth and sixth columns denote the average relative differences which for $\tau=1,3$ correctly indicate that $Y$ is more influential on $X$ than vice versa while for $\tau=10,50$ this is not the case anymore which again is consistent with the construction.}
\label{tab:ARresults}
\end{table}

%% file: BasketballExample.tex
\subsection{Real-world example: Basketball player movement}
We will now apply the dependency measures to the movement of basketball players during a game and quantify influences between players in the same manner as in the previous examples. For this, we use player tracking data captured by the SportVU technology (Stats LLC, Chicago) and publicly provided by Neil Johnson~\cite{johnson} from a game of the 2015/16 NBA season between the Dallas Mavericks and the Cleveland Cavaliers, played on 12th January 2016 in Dallas. The data give the $x$- and $y$ coordinates of each player on the court in 25 time frames per second for most of the 48 minutes of play. The ball data sometimes seem out of sync with the positions of the players and are not always available, therefore they are not used in the analysis. The positions are measured in the unit feet (ft). Using the same data, although from other games, several scientific publications have been made, e.g.,~\cite{bornnWu,bornnDaly} which focus on different problems than this article does.


In basketball, each team has five players on court at all times. Typically, all five players attack or defend simultaneously so that all ten players are in one half of the court for several, around 10 to 20, seconds before moving into the other half. The basketball court has a rectangular shape with a width ($x$ axis) of 94 ft and a length ($y$ axis) of 50 ft. We install a coordinate system whose origin is at the middle point of both axes. If a player is sitting on the bench and not actively participating in the game, we assign to him the coordinate $(-48.5,-27.5)$, which is slightly outside of the court, but exclude these time frames from the analysis.

Figure~\ref{fig:DALvsCLE_JamesBild} shows the distribution of positions of Cleveland player LeBron James depending on whether he is attacking or defending and his position on court over time during the first half of play. We can see that during attack James can mostly be found around the 3-point-line and occasionally closer to the basket, including often at the edge of the small rectangular area around the basket. On defense, he is typically positioned slightly to the left or right of the basket.

\begin{figure}[ht]
\centering
\includegraphics[width=\textwidth]{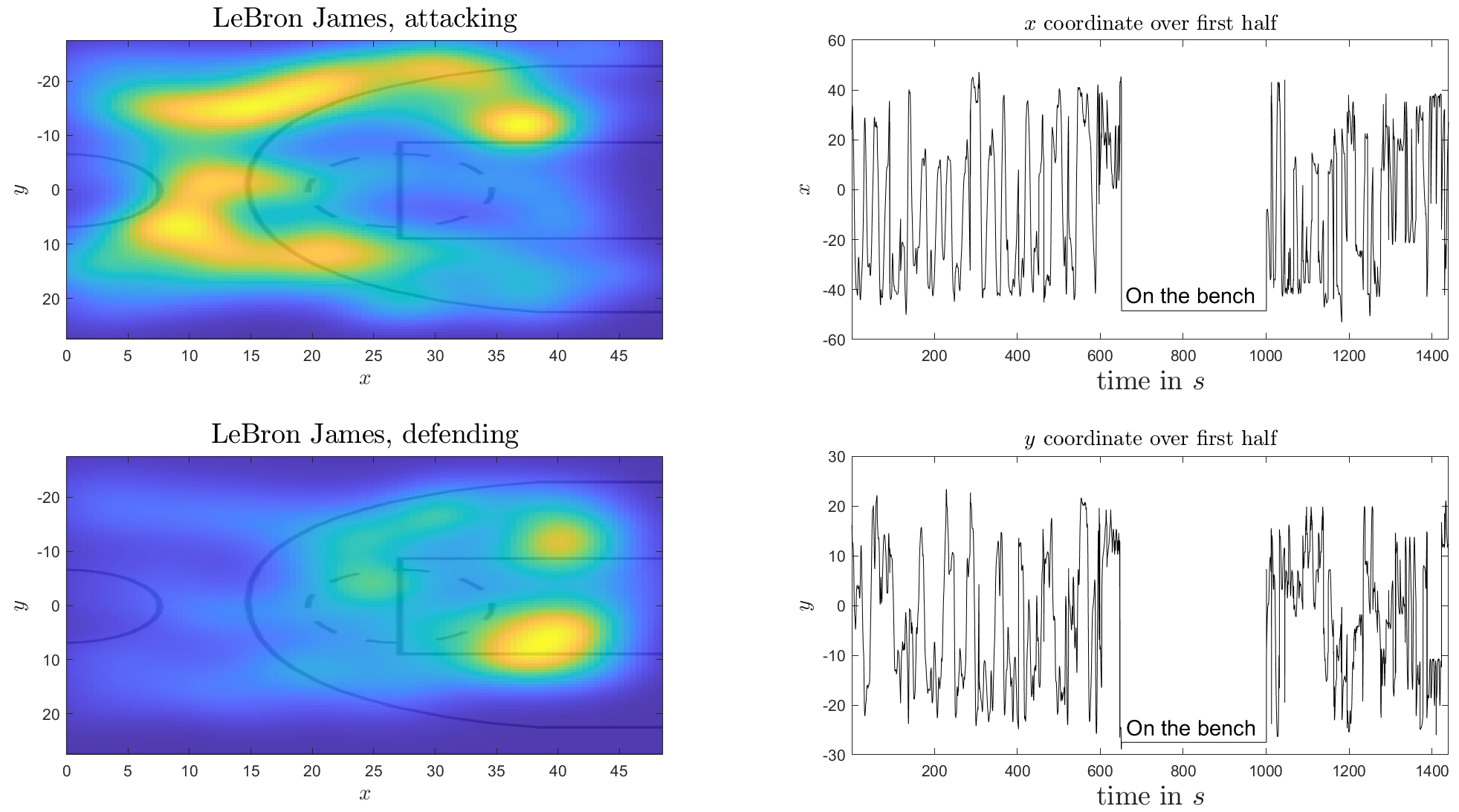}
\caption{Left: distribution of positions of LeBron James during the first half of the game between the Cleveland Cavaliers and the Dallas Mavericks on the 12th January 2016 depending on whether he is in his own team's (defending) or the opponents half (attacking). The visualization was derived using Kernel Density Estimation~\cite{parzen1962}. Right: $x$ and $y$ coordinates over time, measured in ft.}
\label{fig:DALvsCLE_JamesBild}
\end{figure}

\subsubsection*{Applying the dependency analysis to the basketball data}
We now perform the dependency analysis on player coordinate data during the first half of the game. We consider only the ten players in the starting lineups of the teams. For the representation of each two-dimensional position of a player, instead of solving \ref{eq:SPA1}, we choose landmark points in advance and solve \eqref{eq:SPA1} only for $\Gamma$, i.e., solely compute the barycentric coordinates of player coordinates with respect to the landmark points. For clarity of both visualization and numerical computations, we consider only the absolute value of the $x$-coordinates, meaning that we reflect coordinates along the half-court line. For this reason, we can consider only the right half of the court and use the following landmark points,
\begin{equation}
\Sigma = \begin{bmatrix}
    48.5 & 48.5 & 0 & 0 & 20 & 20 & 40\\
    -27.5 & 27.5 & -27.5 & 27.5 & 15 & -15 & 0
\end{bmatrix}
\end{equation}
so that $K_X = K_Y = 7$. The landmark points are depicted in Figure~\ref{fig:basketballLandmarks}.
\begin{figure}[ht]
\centering
\includegraphics[width=0.5\textwidth]{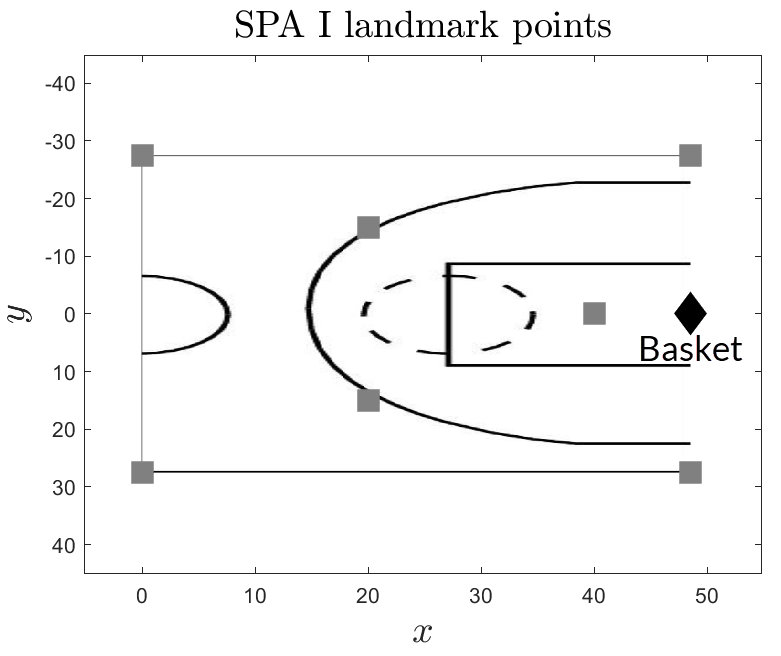}
\caption{Landmark points chosen with respect to which the barycentric coordinates of the player coordinates are computed.}
\label{fig:basketballLandmarks}
\end{figure}
To measure dependencies between each two players, we use those points in time during the game, when both players were on court and compute $\Lambda^{(\tau)}_{XY}$ for each pair of players $X,Y$. We choose $\tau = 1\text{sec}$, so that we investigate the influence of the position of a player $X$ for the position of a player $Y$ one second later.

Note that, as mentioned, basketball games are frequently interrupted for various reason such as fouls or time-outs. We take this into account by defining an event as the part of play between interruptions and denote the number of events considered by $L$. We then construct the training data in the form of multiple short time series, i.e., storing coordinates from the $k$th event, of length $T_k$ seconds, as $\Gamma^X_{k} = [\gamma^X_{k,1},\dots,\gamma^X_{k,T_k-\tau}],\Gamma^Y_{k} = [\gamma^Y_{k,1+\tau},\dots,\gamma^Y_{k,T_k}]$ and for $\Lambda_{XY}$ minimize $\Vert [\Gamma^Y_{1},\dots,\Gamma^Y_{L}] - \Lambda [\Gamma^X_{1},\dots,\Gamma^X_{L}]\Vert_F$ (vice versa for $\Lambda_{YX}$). For the computation of $\Lambda_{XY}$ between two players, we consider only those points in time when both players are on the court. 

Furthermore, we distinguish between which team is attacking since decisions of players should be strongly influenced by whether they are attacking or defending. We therefore define three different scenarios:
\begin{equation*}
    \begin{split}
\text{Dallas attacking: } & \Leftrightarrow \text{ at least 9 players in Cleveland's half}\\
\text{Cleveland attacking: } & \Leftrightarrow \text{ at least 9 players in Dallas' half}\\
\text{Transition: } & \Leftrightarrow \text{ otherwise}
\end{split}
\end{equation*}
We omit the analysis on the transition phase since play is typically rather unstructured during this phase.

The full results of the dependency analysis are, for the sake of visual clarity, shown only in Appendix~\ref{sec:basketballresults}. Here in the main text, selected parts are shown, restricted to the average row variance. We can make the following observations:
\begin{itemize}
    \item As Tables~\ref{tab:DAL_attack_arv} and \ref{tab:CLE_attack_arv} show, according to the dependency measures, with either team attacking, offensive players seem to have the strong often the strongest, influence on their direct opponent, i.e., the opposing player playing the same position (Point Guards (PG) Williams and Irving; Shooting Guards (SG) Matthews and Smith; Small Forwards (SF) Parsons and James; Power Forwards (PF) Nowitzki and Love). Exceptions are Centers (C) Pachulia and Thompson. This is intuitive since in basketball often times each offensive player has one direct opponent that follows him.
    \item We can also see that typically the defending Point Guards and Shooting Guards seem to be more strongly influenced by multiple opposing players than the Power Forwards and Centers. The reason for this could be that PFs and Cs usually are more steadily positioned around the basket while the Guards are more encouraged to actively chase opposing attackers.
    \item We can check if the players of the depending team orient themselves on the positions of the attacking players. To see this, we sum over all dependency values between players from opposing teams, i.e., compute $\sum_{X\text{ in team 1},Y \text{ in team 2}} M_{XY}$ for both dependency measures. We find,
    \FloatBarrier
    \begin{table}[ht]
    \centering
    \begin{tabular}{c|c|c|c}
        Attacking team & Measure & $X$ in DAL, $Y$ in CLE & $X$ in CLE, $Y$ in DAL\\
        \hline
            Dallas  & $\Vert \cdot \Vert_1$ &  54.4 &  52.2\\
            Dallas  & $\nu$ &  0.30 & 0.25\\
            Cleveland  & $\Vert \cdot \Vert_1$ & 51.4 &  51.3\\
            Cleveland & $\nu$ &  0.25 & 0.24
    \end{tabular}
    \end{table}
    \FloatBarrier
    We can see that when Dallas attack, they have a much stronger influence on Cleveland than vice versa according to the average row variance. The Schatten-1 norm does not strongly emphasise that. When Cleveland attack, the cumulated dependencies are very similar to each other. This can be checked using the full results data in Appendix~\ref{sec:basketballresults}. Table~\ref{tab:DAL_attack_arv_rel} also indicates this as it shows that the relative differences are mostly positive for Dallas' players when Dallas attack.
    \item When Cleveland attack, Thompson has large positive relative differences over most other players, except for Pachulia (Table~\ref{tab:CLE_attack_arv_rel}). This could be explained by the fact that Thompson plays below the basket, giving him a lower position radius, and his position is less dependent on spontaneous or set-up moves by his teammates. Pachulia is his direct opponent from whom he might try to separate himself so that his position is in fact influenced by Pachulia.
\end{itemize}
Due to inherent randomness in a sport like basketball, one must be cautious not to overstate these results. This example is meant to showcase how the method presented in this article can be applied to a complex real-world example. The obtained results, however, seem plausible for the explained reasons.

In the basketball context specifically, it would be interesting to see if this approach could be used to detect which player in fact has the strongest influence on the game and whether there are so-far hidden structures as to whose movements certain players seem to pay special attention to. While the presented method might not fully suffice for this yet, the presented example might serve as a starting point to use the dependency measures in a more basketball-tailored way. Ideally, one could use this approach to study otherwise hard to detect patterns of play of a team and develop optimal strategies against them.

\begin{table}[ht]
\begin{tabular}{|c|c|c|c|c|c|} 
 \hline 
From $\downarrow$  to  $\rightarrow$ &Irving (PG) &Smith (SG) &James (SF) &Love (PF) &Thompson (C)\\ \hline
Williams (PG)&\colorbox[rgb]{0.7,0.7,0}{2.73}&\colorbox[rgb]{0.7,0.7,0.642}{1.15}&\colorbox[rgb]{0.7,0.7,0.664}{1.1}&\colorbox[rgb]{0.7,0.7,0.917}{0.48}&\colorbox[rgb]{0.7,0.7,0.808}{0.75}\\ 
Matthews (SG)&\colorbox[rgb]{0.7,0.7,0.885}{0.56}&\colorbox[rgb]{0.7,0.7,0.098}{2.49}&\colorbox[rgb]{0.7,0.7,0.69}{1.04}&\colorbox[rgb]{0.7,0.7,0.764}{0.86}&\colorbox[rgb]{0.7,0.7,1}{0.27}\\ 
Parsons (SF)&\colorbox[rgb]{0.7,0.7,0.605}{1.25}&\colorbox[rgb]{0.7,0.7,0.367}{1.83}&\colorbox[rgb]{0.7,0.7,0.114}{2.45}&\colorbox[rgb]{0.7,0.7,0.881}{0.57}&\colorbox[rgb]{0.7,0.7,0.62}{1.21}\\ 
Nowitzki (PF)&\colorbox[rgb]{0.7,0.7,0.562}{1.35}&\colorbox[rgb]{0.7,0.7,0.772}{0.83}&\colorbox[rgb]{0.7,0.7,0.701}{1.01}&\colorbox[rgb]{0.7,0.7,0.587}{1.29}&\colorbox[rgb]{0.7,0.7,0.718}{0.97}\\ 
Pachulia (C)&\colorbox[rgb]{0.7,0.7,0.623}{1.2}&\colorbox[rgb]{0.7,0.7,0.391}{1.77}&\colorbox[rgb]{0.7,0.7,0.618}{1.21}&\colorbox[rgb]{0.7,0.7,0.668}{1.09}&\colorbox[rgb]{0.7,0.7,0.696}{1.02}\\ 
\hline
\end{tabular}
\caption{\textbf{Dallas attacking: Average row variance multiplied with 100} for $\Lambda_{XY}$ for $X$ from attacking players (Dallas) and $Y$ from defending players (Cleveland).}
\label{tab:DAL_attack_arv}
\vspace{1cm}
\begin{tabular}{|c|c|c|c|c|c|} 
 \hline 
From $\downarrow$  to  $\rightarrow$ &Irving (PG) &Smith (SG) &James (SF) &Love (PF) &Thompson (C)\\ \hline
Williams (PG)&\colorbox[rgb]{0.7,0.7,0.548}{0.11}&\colorbox[rgb]{0.7,0.7,0.504}{0.17}&\colorbox[rgb]{0.7,0.7,0.844}{-0.27}&\colorbox[rgb]{0.7,0.7,0.965}{-0.42}&\colorbox[rgb]{0.7,0.7,1}{-0.47}\\ 
Matthews (SG)&\colorbox[rgb]{0.7,0.7,0.695}{-0.08}&\colorbox[rgb]{0.7,0.7,0.328}{0.4}&\colorbox[rgb]{0.7,0.7,0.412}{0.29}&\colorbox[rgb]{0.7,0.7,0.928}{-0.38}&\colorbox[rgb]{0.7,0.7,0.994}{-0.46}\\ 
Parsons (SF)&\colorbox[rgb]{0.7,0.7,0.726}{-0.12}&\colorbox[rgb]{0.7,0.7,0.374}{0.34}&\colorbox[rgb]{0.7,0.7,0.562}{0.1}&\colorbox[rgb]{0.7,0.7,0.854}{-0.28}&\colorbox[rgb]{0.7,0.7,0.924}{-0.37}\\ 
Nowitzki (PF)&\colorbox[rgb]{0.7,0.7,0}{0.82}&\colorbox[rgb]{0.7,0.7,0.18}{0.59}&\colorbox[rgb]{0.7,0.7,0.082}{0.72}&\colorbox[rgb]{0.7,0.7,0.435}{0.26}&\colorbox[rgb]{0.7,0.7,0.406}{0.3}\\ 
Pachulia (C)&\colorbox[rgb]{0.7,0.7,0.214}{0.55}&\colorbox[rgb]{0.7,0.7,0.159}{0.62}&\colorbox[rgb]{0.7,0.7,0.194}{0.57}&\colorbox[rgb]{0.7,0.7,0.516}{0.16}&\colorbox[rgb]{0.7,0.7,0.475}{0.21}\\ 
\hline
\end{tabular}
\caption{\textbf{Dallas attacking: relative differences of average row variance} for $\Lambda_{XY}$ for $X$ from attacking players (Dallas) and $Y$ from defending players (Cleveland).}
\label{tab:DAL_attack_arv_rel}
\vspace{1cm}
\begin{tabular}{|c|c|c|c|c|c|} 
 \hline 
From $\downarrow$  to  $\rightarrow$ &Williams (PG) &Matthews (SG) &Parsons (SF) &Nowitzki (PF) &Pachulia (C)\\ \hline
Irving (PG)&\colorbox[rgb]{0.7,0.7,0}{2.48}&\colorbox[rgb]{0.7,0.7,0.639}{1.05}&\colorbox[rgb]{0.7,0.7,0.709}{0.9}&\colorbox[rgb]{0.7,0.7,0.424}{1.53}&\colorbox[rgb]{0.7,0.7,0.952}{0.35}\\ 
Smith (SG)&\colorbox[rgb]{0.7,0.7,0.808}{0.67}&\colorbox[rgb]{0.7,0.7,0.306}{1.8}&\colorbox[rgb]{0.7,0.7,0.733}{0.84}&\colorbox[rgb]{0.7,0.7,0.918}{0.43}&\colorbox[rgb]{0.7,0.7,0.961}{0.33}\\ 
James (SF)&\colorbox[rgb]{0.7,0.7,0.59}{1.16}&\colorbox[rgb]{0.7,0.7,0.979}{0.29}&\colorbox[rgb]{0.7,0.7,0.535}{1.29}&\colorbox[rgb]{0.7,0.7,0.746}{0.81}&\colorbox[rgb]{0.7,0.7,0.937}{0.38}\\ 
Love (PF)&\colorbox[rgb]{0.7,0.7,0.828}{0.63}&\colorbox[rgb]{0.7,0.7,0.754}{0.79}&\colorbox[rgb]{0.7,0.7,0.777}{0.74}&\colorbox[rgb]{0.7,0.7,0.652}{1.02}&\colorbox[rgb]{0.7,0.7,1}{0.24}\\ 
Thompson (C)&\colorbox[rgb]{0.7,0.7,0.659}{1.01}&\colorbox[rgb]{0.7,0.7,0.413}{1.56}&\colorbox[rgb]{0.7,0.7,0.571}{1.2}&\colorbox[rgb]{0.7,0.7,0.442}{1.49}&\colorbox[rgb]{0.7,0.7,0.643}{1.04}\\ 
\hline
\end{tabular}
\caption{\textbf{Cleveland attacking: Average row variance multiplied with 100} for $\Lambda_{XY}$ for $X$ from attacking players (Cleveland) and $Y$ from defending players (Dallas)}
\label{tab:CLE_attack_arv}
\vspace{1cm}
\begin{tabular}{|c|c|c|c|c|c|} 
 \hline 
From $\downarrow$  to  $\rightarrow$ &Williams (PG) &Matthews (SG) &Parsons (SF) &Nowitzki (PF) &Pachulia (C)\\ \hline
Irving (PG)&\colorbox[rgb]{0.7,0.7,0.627}{-0.08}&\colorbox[rgb]{0.7,0.7,0.706}{-0.18}&\colorbox[rgb]{0.7,0.7,0.808}{-0.31}&\colorbox[rgb]{0.7,0.7,0.726}{-0.21}&\colorbox[rgb]{0.7,0.7,0.933}{-0.46}\\ 
Smith (SG)&\colorbox[rgb]{0.7,0.7,0.847}{-0.36}&\colorbox[rgb]{0.7,0.7,0.685}{-0.15}&\colorbox[rgb]{0.7,0.7,0.58}{-0.03}&\colorbox[rgb]{0.7,0.7,0.686}{-0.16}&\colorbox[rgb]{0.7,0.7,1}{-0.55}\\ 
James (SF)&\colorbox[rgb]{0.7,0.7,0.541}{0.02}&\colorbox[rgb]{0.7,0.7,1}{-0.55}&\colorbox[rgb]{0.7,0.7,0.8}{-0.3}&\colorbox[rgb]{0.7,0.7,0.507}{0.07}&\colorbox[rgb]{0.7,0.7,0.781}{-0.27}\\ 
Love (PF)&\colorbox[rgb]{0.7,0.7,0.183}{0.47}&\colorbox[rgb]{0.7,0.7,0.068}{0.61}&\colorbox[rgb]{0.7,0.7,0.421}{0.17}&\colorbox[rgb]{0.7,0.7,0.906}{-0.43}&\colorbox[rgb]{0.7,0.7,0.861}{-0.37}\\ 
Thompson (C)&\colorbox[rgb]{0.7,0.7,0}{0.69}&\colorbox[rgb]{0.7,0.7,0.072}{0.61}&\colorbox[rgb]{0.7,0.7,0.251}{0.38}&\colorbox[rgb]{0.7,0.7,0.158}{0.5}&\colorbox[rgb]{0.7,0.7,0.458}{0.13}\\ 
\hline
\end{tabular}
\caption{\textbf{Cleveland attacking: relative differences in the average row variance} for $\Lambda_{XY}$ for $X$ from attacking players (Cleveland) and $Y$ from defending players (Dallas).}
\label{tab:CLE_attack_arv_rel}
\end{table}

%% file: Conclusion.tex
\section{Conclusion}

In this article, a data-driven method for the quantification of influences between variables of a dynamical system was presented. The method deployed the low-cost discretization algorithm Scalable Probabilistic Approximation (SPA) which represents the data points using fuzzy affiliations, respectively, barycentric coordinates with respect to a convex polytope, and estimates a linear mapping between these representations of two variables and forward in time. Two dependency measures were introduced that compute properties of the mapping and admit a suitable interpretation.

Clearly, many methods for the same aim already exist. However, most of them are suited for specific scenarios and impose assumptions on the relations or the dynamics which are not always fulfilled. Hence, this method should be a helpful and directly applicable extension to the landscape of already existing methods for the detection of influences.

The advantages of the method lie in the following: it is purely data-driven and due to the very general structure of the SPA model requires almost no intuition about the relation between variables and the underlying dynamics. This is in contrast to a method such as Granger causality which imposes a specific autoregressive model between the dynamics. Furthermore, the presented method is almost parameter-free since only the numbers of landmark points $K_X$ and $K_Y$ for the representation of points and the time lag $\tau$ have to be specified. Additionally, the dependency measures Schatten-1 norm and average row variance are straightforward to compute from the linear mapping and offer direct interpretation. The capacity of the method to reconstruct influences was demonstrated on multiple examples, including stochastic and memory-exhibiting dynamics.

In the future, it could be worthwhile to find rules for the optimal number of landmark points and their exact positions with respect to the data. Plus, it seems important to investigate how dependencies between variables with differing numbers of landmark points can be compared with the presented dependency measures. Moreover, one could determine additional properties of the matrix $\Lambda_{XY}$ next to the two presented ones. It should also be investigated why in the presented examples the average row variance gave a clearer reconstruction of influences than the Schatten-1 norm and how the latter can be improved. Furthermore, constructing a nonlinear SPA model consisting of the multiplication of a linear mapping with a nonlinear function, as done in \cite{wulkowmSPA}, could give an improved model accuracy and therefore a more reliable quantification of influences. Lastly, since in this article we always expressed variables using a higher number of landmark points, it should be interesting to investigate whether for high-dimensional variables projecting them to a low-dimensional representation using SPA and performing the same dependency analysis is still sensible. This could be of practical help to shed light on the interplay between variables in high dimensions.


\paragraph{Acknowledgements} The author thanks Vikram Sunkara of Zuse Institute Berlin and Illia Horenko from USI Lugano for helpful discussions about the manuscript and the SPA method and Neil Johnson from Monumental Sports \& Entertainment for provision of the basketball data. This work was funded by the Deutsche Forschungsgemeinschaft (German Research Foundation). 

%% file: Appendix.tex
\appendix
\section{Appendix}
\label{sec:Appendix}

\subsection{The dependency measures in light of conditional expectations}
\label{sec:App_condexp}
Assume that $X$ does not contain any information for future states of $\Lambda$, then the entries of $\gamma^X_{t-1}$ should have no influence in the composition of the approximation of $\gamma^Y_t$ by $\Lambda_{XY}\gamma^X_{t-1}$. Thus the columns of $\Lambda_{XY}$ are identical, i.e., $(\Lambda_{XY})_{|1} = \dots (\Lambda_{XY})_{|K_X} =: \lambda$. It then holds
\begin{equation}
    \Lambda_{XY} \gamma^X \equiv (\Lambda_{XY})_{|1} = \dots (\Lambda_{XY})_{|K_X} =: \lambda
    \label{eq:lambdaOneColumn}
\end{equation}
regardless of $\gamma^X$ since $\Lambda_{XY}$ is column-stochastic and each barycentric coordinate is a stochastic vector.

From classical theory on statistics~\cite{spokoiny}, $\lambda$ should be equal to the mean of the time series $\gamma^Y_2,\dots,\gamma^Y_T$ in the limit of infinite data:
\begin{proposition}
Let two time series of length $T$ of barycentric coordinates by given by $\gamma^Y_1,\dots,\gamma^Y_T$ and $\gamma^X_1,\dots,\gamma^X_T$. Assume that 
\begin{equation*}
   \min_{\Lambda^*\in \R^{K_Y\times K_X}} \Vert [\gamma^Y_2 \vert \cdots \vert\gamma^Y_T] - \Lambda  [\gamma^X_1 \vert \cdots \vert\gamma^X_{T-1}] \Vert_F = \min_{\lambda^*\in \R^{K_Y}} \Vert [\gamma^Y_2 \vert \cdots \vert\gamma^Y_T] - [\lambda^*,\dots,\lambda^*] \Vert_F 
\end{equation*}
Then
\begin{equation*}
   \lim_{T\rightarrow \infty} \Lambda_{XY} = \lim_{T\rightarrow \infty}[\frac{1}{T}\sum\limits_{i=1}^T \gamma^Y_i\dots\frac{1}{T}\sum\limits_{i=1}^T \gamma^Y_i].
\end{equation*}
\end{proposition}
\begin{proof}
The minimizer of $\Vert [\gamma^Y_2 \vert \cdots \vert\gamma^Y_T] - [\lambda^*,\dots,\lambda^*] \Vert_F$ is given by the mean $\lambda^* = \frac{1}{T}\sum\limits_{i=1}^T \gamma^Y_i$ for $T\rightarrow \infty$. By assumption of the proposition the product of $\Lambda_{XY} \gamma^X_t$ should therefore be equal to $\lambda^*$ for each $t$. By Eq.~\eqref{eq:lambdaOneColumn}, this is achieved by choosing $\Lambda_{XY} = [\lambda^*,\dots,\lambda^*]$.
\end{proof}

This is consistent with the intuition that if $\gamma^Y_{t}$ is independent of $\gamma^X_{t-1}$, this means that
\begin{equation}
\mathbb{E}_{\mu_Y}[\gamma^Y_{t} |  \gamma^X_{t-1}] = \mathbb{E}_{\mu_Y}[\gamma^Y_{t}].
\end{equation}
Therefore, each column of $\Lambda_{XY}$ should be an approximation of $\mathbb{E}_{\mu_Y}[\gamma^Y_{t}]$. This is naturally given by the statistical mean along a time series.

\subsection{Proofs of propositions in Section~\ref{sec:dependency}}
In this section we write $n$ for $K_Y$ and $m$ for $K_X$ to improve the visual clarity.
\label{sec:proofs}
\begin{lemma}
The maximal Schatten-1 norm of a column-stochastic $(n\times m)$-matrix is $m$.
\label{lem:maxschatten_m}
\end{lemma}
\begin{proof}
Since $\Vert \cdot \Vert_1$ is a norm, the triangle inequality holds and yields
\begin{equation*}
\Vert A+B \Vert_1 \leq \Vert A \Vert_1 + \Vert B \Vert_1.
\end{equation*}
Let there be a finite number of matrices $A_i$ s that $A = \sum\limits_{i} A_i$. It then holds
\begin{equation*}
\Vert A \Vert_1 \leq \sum\limits_{i} \Vert A_i \Vert_1.
\end{equation*}
Note that $A$ can be written as $A = \sum\limits_{i=1}^n \sum\limits_{j=1}^m \tilde{A}^{ij}$ where
\begin{equation*}
(\tilde{A}^{ij})_{kl} \left\{
\begin{array}{ll}
A_{ij} & (k,l) = (i,j) \\
0 & \, \textrm{else} \\
\end{array}
\right.
\end{equation*}
A matrix with only one non-zero entry $a$ has only one non-zero singular value that is equal to $a$. This means that $\Vert \tilde{A}^{ij} \Vert_1 = A_{ij}$. Furthermore, since $A$ is column-stochastic, it holds that $\sum_{i=1}^n A_{ij} = 1$. Thus,
\begin{equation*}
\Vert A \Vert_1 \leq \sum\limits_{i=1}^n \sum\limits_{j=1}^m \Vert \tilde{A}^{ij} \Vert = \sum\limits_{i=1}^n \sum\limits_{j=1}^m A_{ij} = m.
\end{equation*}
\end{proof}
\textbf{Proof of Proposition~\ref{lem:schattenminimalKY>KX}}:
\begin{proposition}
The Schatten-1 norm of a column-stochastic $(n\times m)$-matrix $A$ with $n \geq m$ obtains the maximal value $m$ if deletion of $n-m$ rows of $A$ yields an $m \times m$ permutation matrix.
\end{proposition}
\begin{proof}
Let $A$ be of the form described in the proposition. Then by deletion of $n-m$ rows we can derive a permutation matrix $P$. For those matrices, it holds that 
\begin{equation*}
PP^T = Id.
\end{equation*}
Thus, the singular values of $P$, which are the square roots of the eigenvalues of $PP^T$ are given by the square roots of the eigenvalues of the $m \times m$ identity matrix. These are given by $\sigma_1 = \dots = \sigma_m = 1$. Thus, $\Vert P \Vert_1 = m$.\\
All deleted rows must be identical to the zero-vector of length $m$, since the column sums of $A$ have to be equal to $1$ and the column sums of $P$ are already equal to $1$. Therefore, the singular values of $P$ are equal to the singular values of $A$ and their sum is equal to $m$ because the sum of singular values of a matrix cannot shrink by adding zero rows. This is because
\begin{equation*}
AA^T = 
\begin{pmatrix}
P\\
0
\end{pmatrix}
\begin{pmatrix}
P^T & 0
\end{pmatrix}
= \begin{pmatrix}
PP^T & 0\\
0 & 0
\end{pmatrix},
\end{equation*}
whose eigenvalues are the eigenvalues of $PP^T$, i.e. the singular values of $P$, and additional zeros. Thus, the sum of singular values does not change. Since $m$ is the maximal value for $\Vert A \Vert_1$ by Lemma~\ref{lem:maxschatten_m}, it holds that $\Vert A \Vert_1 = \Vert P \Vert_1 = m$.
\end{proof}

\textbf{Proof of Proposition~\ref{lem:schattenminimizer}}:
\begin{proposition}
The Schatten-1 norm of a column-stochastic $(n\times m)$-matrix $A$ is minimal if and only if $A_{ij} \equiv \frac{1}{n}$ and in this case is equal to $\sqrt{\frac{m}{n}}$.
\end{proposition}
\begin{proof}
If all entries of $A$ are given by $\frac{1}{n}$, then 
\begin{equation}
(AA^T)_{ij} = \frac{m}{n^2} \quad \text{for all } i,j = 1,\dots, n.
\end{equation}
Then $AA^T$ has exactly one non-zero eigenvalue, since it is a rank-$1$ matrix. This is equal to $\frac{m}{n}$ (corresponding to the eigenvector $(1,\dots,1)^T$), since
\begin{equation*}
AA^T \begin{pmatrix}
1\\
\vdots\\
1
\end{pmatrix} = \frac{m}{n^2}\begin{pmatrix}
1 & \dots & 1\\
\vdots & & \vdots\\
1 & \dots & 1
\end{pmatrix}
\begin{pmatrix}
1\\
\vdots\\
1
\end{pmatrix}
=
\frac{m}{n^2} \begin{pmatrix}
n\\
\vdots\\
n
\end{pmatrix}
= \frac{m}{n}\begin{pmatrix}
1\\
\vdots\\
1
\end{pmatrix}.
\end{equation*}
Singular values of a matrix $A$ are square roots of the eigenvalues of $AA^T$. The square root of the eigenvalue above, which is the only positive singular value, is then $\sqrt{\frac{m}{n}}$. This yields 
\begin{equation*}
\Vert A \Vert_1 = \sqrt{\frac{m}{n}}
\end{equation*}
The reason there cannot be a column-stochastic matrix $B$ with $\Vert B \Vert_1 < \sqrt{\frac{m}{n}}$ is the following: It holds that
\begin{equation*}
\Vert B \Vert_1 \geq \Vert B \Vert_2 \quad \text{where }\Vert B \Vert_2 := \sqrt{\sum\limits_{i=1}^{min(n,m)} \sigma_i^2} \text{ is the } \textit{Schatten-2-norm}.
\end{equation*}
because it is well-known that the standard 1-norm of a Euclidean vector is bigger or equal to its 2-norm so that $\Vert (\sigma_1,\dots,\sigma_{min(n,m)})\Vert_1 \geq \Vert(\sigma_1,\dots,\sigma_{min(n,m)})\Vert_2$. For the Schatten-2 norm, it holds 
\begin{equation}
\Vert B \Vert_2 = \Vert B \Vert_F := \sqrt{\sum\limits_{i=1}^n \sum\limits_{j=1}^m B_{ij}^2},
\end{equation}
which, by a classic linear algebra result, is the Frobenius norm of $B$. Note that, if $B_{ij} = \frac{1}{n}$ for all $i,j$, then
\begin{equation}
\Vert B \Vert_2  = \Vert B \Vert_F = \sqrt{\sum\limits_{i=1}^n \sum\limits_{j=1}^m \frac{1}{n^2}} = \sqrt{\frac{nm}{n^2}} = \sqrt{\frac{m}{n}}.
\end{equation}
Assume that there exists a $B_{ij} < \frac{1}{n}$, e.g., $B_{ij} = \frac{1}{n} - \delta$. Then since $\sum\limits_{i=1}^n \sum\limits_{j=1}^m B_{ij} = 1$, there also exists $B_{kj} > \frac{1}{n}$. Immediately, this increases the sum over the squared entries of $B$, simply because $( a + \delta)^2 + (a- \delta)^2 = 2a^2 + 2\delta^2 > 2a^2$ for $\delta > 0$. Therefore, $\Vert B \Vert_F$ is minimal only for the choice given above. As a consequence, $\Vert B \Vert_2$, is minimal in this case, too, and thus this is the only minimizer of $\Vert B \Vert_1$.\\
\end{proof}

\textbf{Proof of Proposition~\ref{lem:variancemaximizer_KY>KX}}:
\begin{proposition}
The average row variance of a column-stochastic $(n\times m)$-matrix $A$ with $n=m$ is maximal and equal to $\frac{1}{m}$ if it is a permutation matrix. 
\end{proposition}
\begin{proof}
For the case $K_Y = K_X$, i.e., $n=m$, this is straightforward: the variance of a row that contains only values between $0$ and $1$ is maximized if exactly one value is $1$ and all other entries are $0$. Each column, due to its being stochastic, can only have one $1$. For $n\geq m$ one therefore has to distribute the ones into different rows and columns, matching the statement of the proposition. The maximal value is hence equal to the variance of an $m$-dimensional unit vector.
Since the mean of this vector is equal to $\frac{1}{m}$, the variance is
\begin{equation}
\begin{split}
    \frac{1}{m-1}( (1- \frac{1}{m})^2 + \sum_{i=1}^{m-1} \frac{1}{m^2})&= \frac{1}{m-1}((\frac{m-1}{m})^2 + \sum_{i=1}^{m-1} \frac{1}{m^2} )\\
    &= \frac{1}{m-1}(\frac{m^2 - 2m +1}{m^2} +  \frac{m-1}{m^2})\\
    &= \frac{1}{m-1} \frac{m^2 - m}{m^2}\\
    &= \frac{1}{m}.
\end{split}
\end{equation}
\end{proof}

\textbf{Proof of Proposition~\ref{lem:varianceminimizer}}:
\begin{proposition}
The average row variance of a column-stochastic matrix $A$ obtains the minimal value $0$ if and only if all columns are equal to each other.
\end{proposition}
\begin{proof}
Trivially, if all columns in $A$ are identical, then the average row variance of $A$ is $0$. Since the variance is always non-negative by definition, this is the minimal value. If at least two values differ in a column, then the average row variance immediately becomes positive.
\end{proof}

\subsection{Differences to related practices}
\label{sec:differences}
\paragraph{Simple linear correlations}
The presented measures might seem strongly related to the computation of the linear correlation
\begin{equation}
    C(X,Y) = \frac{1}{T-1}\sum_{t=2}^{T} (X_{t-1}-\bar{X}) (Y_{t}-\bar{Y})^T
\end{equation}
(for $\tau = 1$) where $\bar{X},\bar{Y}$ are the component-wise averages. However, $C$ can only detect \textit{global} linear patterns between $X$ and $Y$. In contrast, we transform points into a higher-dimensional space by expressing them by barycentric coordinates with $K > D$. While we still determine a linear operator between the variables, given by $\Lambda_{XY}^{(\tau)}$, this is a \textit{local} approximation of a potentially nonlinear function, denoted earlier by $\textbf{w}$. Furthermore, upon perturbations to the function $\textbf{w}$, $\Lambda_{XY}^{(\tau)}$ should react in a nonlinear way by construction of the SPA~II problem. The dependency measures then are nonlinear functions on $\Lambda_{XY}^{(\tau)}$. In the examples that will follow, using linear correlations could generally not uncover unidirectional dependencies while our measures were able to do so.

\paragraph{Granger causality}
A prominent method to measure the influence of one variable on another is by employing the Granger causality framework~\cite{granger,kirchgaessner}. It functions by determining two models of the form
\begin{equation}
\begin{split}
    Y_{t} &= f(X_{t-1},\dots,X_{t-p},Y_{t-1},\dots,Y_{t-q})\\
    Y_{t} &= g(Y_{t-1},\dots,Y_{t-q})
    \end{split}
\end{equation}
from training data and using them to compute subsequent values of $Y_t$ on testing data which was not used for training. The prediction errors of $f$ and $g$ are then compared. If $f$, which uses information of $X$, gives a significantly better prediction error, then it is deduced that $X$ influences $Y$.

Typical model forms for $f$ and $g$ are linear autoregressive models~\cite{brockwell} which are described in more detail in the next section. It is pointed out in \cite{sugihara} that using past terms of $X$ and $Y$ can constrain the interpretability of the result, since if $Y$ forces $X$ information about $Y$ is stored in past terms of $X$ due to the delay-embedding theorem of Takens~\cite{takens} (please see \cite{sugihara} including its supplement for details). Then if $Y$ can be predicted from past terms of $X$ it actually is a sign that $Y$ forces $X$, not vice versa. In \cite{sugihara} examples are shown where the Granger causality method fails to detect influences between variables. This makes the interpretation of the Granger results more difficult.

This effect does not occur when dispensing of the past terms and instead fitting models $Y_{t} = f(X_{t-1},Y_{t-1})$ and $Y_{t} = g(Y_{t-1})$. However, in systems which are highly stochastic or chaotic, meaning that from similar initial conditions diverging trajectories emerge, even an accurate model can be prone to give weak prediction errors. In such cases the prediction error often times has limited meaning.

Furthermore, even if $X$ influences $Y$, one has to select a suitable model family for $f$ and $g$ so that this actually shows. The selection of the model family can be a challenging task of its own.

Nevertheless, Granger causality can be a strong tool for the detection of influences, e.g., as shown for a Gene Regulatory Network in ~\cite{sincerities}.

\paragraph{Discretization by boxes instead of landmark points}
Earlier the similarity between the model constructed solving \eqref{eq:SPA2} and Markov State Models (MSMs) was mentioned. In MSMs, one discretizes the state space into boxes and statistically estimates the transition probabilities of the state of a dynamical system between the boxes, typically by considering the relative frequencies of transitions. One then obtains a column-stochastic transition matrix that contains these relative frequencies. In the same manner, one could compute this matrix for the frequencies that a variable $Y$ is in a certain box at time $t+\tau$ given that a variable $X$ is in a certain box at time $t$ and apply the dependency measures to this transition matrix to assess how meaningful the information about a variable $X$ is for the future value of $Y$. Fixing the edge length of each box, the number of boxes increases exponentially with the dimension of points while one generally requires an appropriately fine discretization of the state space to obtain meaningful results, yielding an high number of boxes even for moderate dimensions of the data. One then requires very long time series for a robust estimation of the transition probabilities. The advantage in the SPA~I representation of points is that we can derive a column-stochastic matrix but can maintain a lossless representation of points with $K\sim D$ while $K > D$ is the only prerequisite.

\subsection{Full basketball dependency results}
\label{sec:basketballresults}
Order of players in all rows and columns: (Dallas) Williams, Matthews, Parsons, Nowitzki, Pachulia, (Cleveland) Irving, Smith, James, Love, Thompson.\\
Dependencies from one player to himself are omitted, therefore the zeros are placed on the diagonals.

\textbf{Dallas attacking}
\begin{equation*}
M_{\Vert \cdot \Vert_1} = \begin{pmatrix} 
0&2.31&2.47&1.76&1.7&2.94&2.29&2.27&1.83&1.9\\\ 
2.3&0&2.09&1.71&2.05&1.9&2.78&2.07&2.04&1.55\\\ 
2.63&2.35&0&1.76&2&2.23&2.52&2.77&1.85&2.13\\\ 
2.05&1.96&1.79&0&1.76&2.05&2.02&1.99&2.17&1.94\\\ 
2.14&2.49&2.2&1.57&0&2.27&2.47&2.2&2.18&2.08\\\ 
2.84&1.92&2.33&1.51&1.82&0&2.25&2.23&1.96&1.99\\\ 
2.13&2.48&2.22&1.66&1.88&2.08&0&2.39&2.29&1.75\\\ 
2.42&1.99&2.71&1.58&1.8&2.55&2.02&0&2.13&1.87\\\ 
2.1&2.32&1.97&2.03&2.07&2.37&2.72&2.09&0&1.73\\\ 
2.31&1.77&2.59&1.85&1.96&2.52&2.22&2.38&1.78&0\\\ 
\end{pmatrix}
\end{equation*}

\begin{equation*}
   M_{\nu} = 0.01\cdot
\begin{pmatrix} 
0&1.51&1.71&0.52&0.49&2.73&1.15&1.1&0.48&0.75\\\ 
1.37&0&0.87&0.53&1.24&0.56&2.49&1.04&0.86&0.27\\\ 
1.89&1.43&0&0.66&0.86&1.25&1.83&2.45&0.57&1.21\\\ 
0.89&0.79&0.41&0&0.57&1.35&0.83&1.01&1.29&0.97\\\ 
1.13&1.74&1.1&0.27&0&1.2&1.77&1.21&1.09&1.02\\\ 
2.42&0.6&1.41&0.24&0.54&0&1.08&1.23&0.69&0.76\\\ 
0.96&1.5&1.21&0.34&0.68&0.85&0&1.57&1.27&0.62\\\ 
1.51&0.74&2.22&0.29&0.52&1.67&0.84&0&1.38&0.57\\\ 
0.83&1.37&0.79&0.96&0.92&1.38&2.55&1.06&0&0.56\\\ 
1.41&0.51&1.92&0.68&0.81&1.95&1.13&1.69&0.49&0\\\ 
\end{pmatrix}
\end{equation*}

\textbf{Cleveland attacking}
\begin{equation*}
M_{\Vert \cdot \Vert_1} = 
\begin{pmatrix} 
0&2.04&2.71&1.99&1.78&2.89&2.19&2.11&1.62&1.62\\\ 
2.05&0&2.26&2.22&1.9&2.28&2.54&1.81&1.64&1.82\\\ 
2.49&2.22&0&2.08&1.77&2.21&2.1&2.53&1.89&1.96\\\ 
2.01&2.5&1.89&0&1.97&2.51&1.81&1.97&2.49&1.99\\\ 
2.04&2.42&2&2&0&1.9&1.95&1.84&1.72&2.02\\\ 
2.89&2.25&2.01&2.37&1.66&0&2.41&2.52&2.08&1.75\\\ 
1.97&2.46&2.07&1.76&1.6&1.96&0&2.14&2.03&1.82\\\ 
2.11&1.66&2.27&1.91&1.65&2.26&2.01&0&1.89&1.78\\\ 
1.79&1.99&1.99&2.09&1.57&1.88&2.34&2.1&0&1.63\\\ 
2.13&2.4&2.33&2.26&2.11&2.23&2.33&2.29&1.86&0\\\ 
\end{pmatrix}
\end{equation*}

\begin{equation*}
M_{\nu} = 0.01\cdot
\begin{pmatrix} 
0&0.97&2.27&0.93&0.66&2.71&1.04&1.13&0.33&0.31\\\ 
0.83&0&1.21&1.13&0.89&1.28&2.13&0.64&0.31&0.61\\\ 
1.58&1.24&0&1.14&0.58&1.29&0.86&1.83&0.61&0.74\\\ 
0.7&1.45&0.49&0&0.75&1.93&0.51&0.76&1.79&0.75\\\ 
0.85&1.54&0.83&0.66&0&0.65&0.73&0.53&0.39&0.91\\\ 
2.48&1.05&0.9&1.53&0.35&0&1.81&1.81&0.95&0.45\\\ 
0.67&1.8&0.84&0.43&0.33&0.71&0&1.11&0.85&0.64\\\ 
1.16&0.29&1.29&0.81&0.38&1.28&0.66&0&0.57&0.6\\\ 
0.63&0.79&0.74&1.02&0.24&0.58&1.42&1.01&0&0.39\\\ 
1.01&1.56&1.2&1.49&1.04&1.31&1.19&1.57&0.78&0\\\ 
\end{pmatrix}
\end{equation*}